\documentclass[12pt,reqno]{amsart}

\usepackage[active]{srcltx}

\usepackage{latexsym}
\usepackage{graphicx}

\renewcommand{\Re}{{\operatorname{Re}\,}}
\renewcommand{\Im}{{\operatorname{Im}\,}}

\renewcommand{\epsilon}{\varepsilon}

\newcommand{\szego}{Szeg\H{o} }

\newcommand{\kahler}{K\"ahler }

\newcommand{\PP}{{\mathbb P}}

\newcommand{\R}{{\mathbb R}}
\newcommand{\C}{{\mathbb C}}

\newcommand{\Z}{{\mathbb Z}}

\newcommand{\Ss}{{\mathbb S}}
\newcommand{\Hh}{{\mathbb H}}

\newcommand{\CP}{\C\PP}

\newcommand{\dbar}{\bar\partial}
\newcommand{\ddbar}{\partial\dbar}

\newcommand{\half}{{\textstyle \frac 12}}

\renewcommand{\phi}{\varphi}

\newcommand{\acal}{\mathcal{A}}
\newcommand{\bcal}{\mathcal{B}}
\newcommand{\ccal}{\mathcal{C}}
\newcommand{\dcal}{\mathcal{D}}

\newcommand{\hcal}{\mathcal{H}}

\newcommand{\lcal}{\mathcal{L}}

\newcommand{\ocal}{\mathcal{O}}
\newcommand{\pcal}{\mathcal{P}}

\newcommand{\scal}{\mathcal{S}}

\newcommand{\ucal}{\mathcal{U}}

\newtheorem{theo}{{\sc Theorem}}[section]

\newtheorem{maintheo}{{\sc Theorem}}
\newtheorem{conj}{{\sc Conjecture}}
\newtheorem{maindefn}{{\sc Definition}}

\newtheorem{mainprop}{{\sc Proposition}}
\newtheorem{cor}[theo]{{\sc Corollary}}
\newtheorem{maincor}[maintheo]{{\sc Corollary}}
\newtheorem{lem}[theo]{{\sc Lemma}}
\newtheorem{prop}[theo]{{\sc Proposition}}

\newenvironment{rem}{\medskip\noindent{\it Remark:\/} }{\medskip}

\newtheorem{defn}[theo]{{\sc Definition}}

\title[Pluri-potential  theory on Grauert tubes   ]
{Pluri-potential  theory on Grauert tubes of  real analytic
Riemannian manifolds, I  }

\address{Department of Mathematics, Northwestern  University, Evanston, IL 60208, USA}

\email{zelditch@math.northwestern.edu}

\thanks{Research partially supported by NSF grant  \# DMS-0904252.}

\date{\today}

\begin{document}
\begin{abstract} Analogues of the some basic  notions of pluri-potential theory on strictly pseudo-convex domains in $\C^m$ are developed for
Grauert tubes $M_{\tau}$  in complexifications of real analytic
Riemannian manifolds $(M, g)$. In particular, the normalized
logarithm of the  complexified spectral projector
$\Pi_{I_{\lambda}}^{\C}(\zeta, \bar{\zeta})$ is the  analogue of
the Siciak-Zaharjuta extremal pluri-subharmonic function. It is
shown that $\frac{1}{\lambda} \log \Pi_{I_{\lambda}}^{\C}(\zeta,
\bar{\zeta}) \to \sqrt{\rho}(\zeta)$, where $\sqrt{\rho}$ is the
Grauert tube function. We give several applications to analytic
continuations of eigenfunctions: to norm estimates, triple product
integrals and to complex nodal sets.

\end{abstract}
 \maketitle


In the study of eigenfunctions  of the Laplacian $\Delta_g$ on a
real analytic Riemannian manifold $(M, g)$ of dimension $m$, it is
often useful to analytically continue an orthonormal basis
$\{\phi_{\lambda_j}\}$   of eigenfunctions,
$$\Delta_g \phi_{\lambda_j} = \lambda_j^2\; \phi_{\lambda_j}, \;\;\;
\langle \phi_{\lambda_j}, \phi_{\lambda_k} \rangle = \delta_{jk},
\;\; (\lambda_0 = 0 < \lambda_1 \leq \lambda_2 \leq \cdots),
$$ into the
complexification $M_{\C}$ of $M$. As recalled in \S \ref{GRAUERT},
eigenfunctions admit analytic continuations
$\phi_{\lambda_j}^{\C}$ to a maximal uniform 'Grauert tube'
\begin{equation} \label{MTAU} M_{\tau} = \{\zeta \in   M_{\C}, \sqrt{\rho}(\zeta) < \tau
\} \end{equation}  independent of $\lambda_j$, where the radius is
measured by the Grauert tube function $\sqrt{\rho} (\zeta) $
corresponding to $g$ (see \S \ref{GRAUERT}:  \cite{LS1,GS1}).
Analytic continuation of eigenfunctions and spectral projections
(\ref{CXSPMa})-(\ref{TCXSPM})  to Grauert tubes  have applications
to nodal geometry \cite{DF,Lin,Z3,TZ,RZ}, analytic wave front sets
\cite{Leb,GLS}, tunnelling estimates \cite{HS,Mar},
Paley-Wiener theorems \cite{G}, invariant triple products
\cite{Sar,BR}, random waves \cite{Z2} and Agmon estimates for
eigenfunctions in the classically forbidden region (see
e.g.\cite{To}).

 Grauert tubes are strictly
pseudo-convex Stein manifolds,  and in some ways are analogous to
strictly pseudo-convex domains in $\C^m$ and to Hermitian unit
bundles in negative line bundles. The purpose of this article and
its sequel \cite{Z1}  is to extend to Grauert tubes some of the
basic notions and results of PSH (pluri-subharmonic) function
theory on stricty pseudo-convex domains in $\C^n$ (cf.
\cite{K,BL}),  and their recent generalization of this theory to
\kahler manifolds in \cite{GZ}. The basic  theme is to use
analytic continuations of eigenfunctions $\{\phi_{\lambda_j}^{\C}
\}$ in place of holomorphic  polynomials of degree $\sim
\lambda_j$  on $\C^m$ or holomorphic sections of line bundles of
degree $\sim \lambda_j$ over a \kahler manifold. The  primary
objects are the analytic continuations of the spectral projections
kernels of $\Delta_g$,
\begin{equation}\label{CXSPMa}   \Pi_{I_{\lambda}}^{\C}(\zeta, \bar{\zeta}) =
\sum_{j: \lambda_j \in  I_{\lambda}}
|\phi_{\lambda_j}^{\C}(\zeta)|^2,
\end{equation}  which are of exponential growth, and  their  `tempered' analogues,
\begin{equation}\label{TCXSPM}   P_{ I_{\lambda}}^{\tau}(\zeta, \bar{\zeta}) =
\sum_{j: \lambda_j \in  I_{\lambda}} e^{- \tau \lambda_j}
|\phi_{\lambda_j}^{\C}(\zeta)|^2, \;\; (\sqrt{\rho}(\zeta) \leq
\tau),
\end{equation}
where $I_{\lambda} $ could be a short interval  $[\lambda, \lambda
+ 1]$ of frequencies or a long window $[0, \lambda]$. In this article, we
only consider long windows $I_{\lambda} = [0, \lambda]$ while in \cite{Z1}
we refine the results to short windows using the long time behavior of the complexified
geodesic flow.  The tempered
kernels $P_{ I_{\lambda}}^{\tau}(\zeta, \bar{\zeta})$ are in some
ways analogous to the `density of states function' or Bergman
kernel on the diagonal in the setting of positive line bundles
over \kahler manifolds \cite{Z4}. We gave some initial results on
these kernels in \cite{Z2}, by somewhat different methods.

A basic notion in PSH theory is that of maximal PSH functions
satisfying bounds and the (non-obviously) equivalent
Siciak-Zaharjuta extremal PSH functions. We define a Grauert tube
analogue of the Siciak-Zaharjuta extremal function and show in
Theorem \ref{SICIAK} that it is the same as the Grauert tube
function. The proof is to relate both to the complexified spectral
projections \eqref{CXSPMa}which are defined in terms of
eigenfunctions.  The proof only requires a one term local Weyl law
(see Theorem \ref{PTAULWL}), which also gives improvements on the
pointwise bounds on complexified eigenfunctions in \cite{GLS}. The
result can be improved to a rather interesting two term Weyl law
of Safarov-Vasilliev type \cite{SV}; this  is carried out in the
sequel \cite{Z1}.

This article also contains  a general type of result on integrals
of triple products of eigenfunctions (Proposition
\ref{TRIPLEPROD}). The precise results depend on the radius of the
maximal Grauert tube. We point out that there are two possible
definitions (see Definition \ref{MAXGRAU}), an analytic maximal radius and a geometric maximal
radius; in \S \ref{MAX} we sketch a proof that these radii are the same.
Finally, we build on \cite{AS} to
give a result on limit distribution of complex zeros of
eigenfunctions on locally symmetric manifolds of non-positive
curvature.

In keeping with the nature of this symposium, this article is
partly expository. In particular, we review the construction of
the Hadamard parametrix of the wave kernel and its holomorphic
extension to Grauert tubes. We also illustrate the issues and
notions with examples from surfaces of constant curvature. We thank
J. Sj\"ostrand and the referee for helpful discussions/comments on the material.

\subsection{A Siciak-Zaharjuta extremal function for Grauert
tubes}

Before defining the analogues, let us  first  recall the
definitions of relative   maximal or extremal PSH
 functions satisfying bounds on a pair $E \subset \Omega \subset
 \C^m$ where $\Omega$ is a bounded open set.
 There are two definitions:

\begin{itemize}

\item  The pluri-complex Green's function relative to a subset $E
\subset \Omega$, defined \cite{Br,Sic} as the upper
semi-continuous regularization $V_{E, \Omega}^*$ of

$$\begin{array}{lll} V_{E, \Omega}(z) & = &  \sup\{u(z): u \in PSH(\Omega), u|_{E}  \leq 0, u
|_{\partial \Omega} \leq 1\}
\end{array}.$$

\item  The Siciak-Zaharjuta  extremal function relative to $E
\subset \Omega$, defined by
$$\log \Phi_E^N(\zeta) = \sup\{ \frac{1}{N} \log |p_N(\zeta)|: p \in \pcal_E^N \}, \;\;\; \log \Phi_E = \limsup_{N \to \infty} \log \Phi_E^N,$$
where $\pcal_E^N = \{p \in \pcal^N: ||p||_E \leq 1, \;\;
||p||_{\Omega} \leq e^N\}. $

\end{itemize}
Here, $||f||_E = \sup_{z \in E} |f(z)|$ and $\pcal^N$ denotes the
space of all complex analytic polynomials of degree $N$.  Siciak
proved  that $\log \Phi_E = V_E$ (see \cite{Sic2} Theorem 1,  and
\cite{K}, Theorem 5.1.7).  Intuitively, there are enough
polynomials that one can obtain the sup by restricting to
polynomials.

There are analogous definitions in the case of unit co-disc
bundles in the dual of a positive holomorphic Hermitian line
bundle $L \to M$ over a \kahler manifold. In the case of $\CP^n$,
 one defines $$V_K(z) =
\sup \{u(z)\colon u \in \lcal, u\leq 0 \; \mbox{on}\; K \}$$ where
$\lcal$ denotes the Lelong class of all global plurisubharmonic
(PSH) functions $u$ on $\C^n$ with $u(z) \leq c_u +\log {(1+
|z|)}$. We refer to \cite{GZ} for further information in the
\kahler setting.

We now define an analogue of the Siciak-Zaharjuta extremal
function for Grauert tubes in the special case where $E = M$, the
underlying real manifold. A generalization to other sets $E
\subset M_{\tau}$ is discussed in \S \ref{GENSICIAK}.
 The
 Riemannian  analogue of $\pcal^N$ is the space
$$\hcal^{\lambda} = \{p =  \sum_{j: \lambda_j \in I_{\lambda}} a_j
\phi_{\lambda_j}^{\C}, \;\; a_1, \dots, a_{N(\lambda)} \in \R  \}
$$ spanned by the eigenfunctions with `degree' $\lambda_j \leq
\lambda$. Here, $N(\lambda) = \#\{j: \lambda_j \in I_{\lambda}\}$.
As above, we could let $I_{\lambda} = [0, \lambda]$  or
$I_{\lambda} = [\lambda, \lambda + c]$ for some $c > 0$. It is
simpler to work with $L^2$ based norms than sup norms, and so we
define
$$ S \hcal^{\lambda}_M = \{\psi =  \sum_{j: \lambda_j \leq \lambda} a_j
\phi_{\lambda_j}^{\C}, \;\;  \sum_{j =
1}^{N(\lambda)} |a_j|^2 = 1 \}.
$$

\begin{maindefn}

The Riemannian Siciak-Zaharjuta  extremal function (with respect
to the real locus $M$) is defined by:
\begin{equation} \label{SICIAKDEF}\left\{ \begin{array}{l}  \log \Phi_M^{\lambda}(\zeta) =
\sup\{ \frac{1}{\lambda} \log |\psi(\zeta)|: \psi \in S
\hcal_M^{\lambda} \}, \\ \\ \log \Phi_M = \limsup_{\lambda \to
\infty} \log \Phi_M^{\lambda}. \end{array} \right.  \end{equation}
\end{maindefn}

\begin{rem}

  One could define the analogous notion for any  set
$E \subset M_{\tau}$, with
$$S \hcal^{\lambda}_E = \{p  \in \hcal^{\lambda}, ||p||_{L^2(E)} \leq 1 \}.
$$ But we only discuss the results  for $E = M$ (see \S
\ref{GENSICIAK} for comments on the general case). \medskip


\end{rem}

One could also define the pluri-complex Green's function of
$M_{\tau}$ as follows:

\begin{maindefn}  Let $(M, g)$ be a real analytic Riemannian manifold,  let $M_{\tau}$
be an open Grauert tube, and let  $E \subset M_{\tau}$.  The
Riemannian pluri-complex Green's function with respect to $(E,
M_{\tau}, g)$ is defined by
$$V_{g, E, \tau}(\zeta) =  \sup\{u(z): u \in PSH(M_{\tau}), u|_{E}  \leq 0, u
|_{\partial M_{\tau}} \leq \tau\}.  $$
\end{maindefn}
It is obvious that $V_{g, M, \tau}(\zeta) \geq \sqrt{\rho}(\zeta)$
and it is almost standard that $V_{g, M, \tau}(\zeta) =
\sqrt{\rho}(\zeta)$. See Proposition 4.1 of \cite{GZ} or Corollary
9 of \cite{BT2}. The  set $M = (\sqrt{\rho})^{-1}(0)$ is often
called the center. As  proved in \cite{LS1},  there are no smooth
exhaustion functions solving the exact HCMA (Theorem 1.1). Hence
$u$ must be singular on its  minimum set. In \cite{HW} it is
proved that the minimum  set of strictly PSH function  is totally
real.

\subsection{Statement of results}

Our first results concern the logarithmic asymptotics of the
complexified spectral projections.

\begin{maintheo}\label{SICIAK}  (see also \cite{Z4})  Let $I_{\lambda} = [0, \lambda]$. Then

\begin{enumerate}

\item  $\log \Phi^{\lambda}_M(\zeta) = \frac{1}{\lambda} \log
\Pi^{\C}_{I_{\lambda}}(\zeta, \bar{\zeta}). $

 \item  $\log \Phi_M = \lim_{\lambda \to
\infty} \log \Phi_M^{\lambda} =  \sqrt{\rho}. $

 \end{enumerate}  \end{maintheo}

To prove the Theorem, it is  convenient to study the tempered
spectral projection measures (\ref{TCXSPM}), or in differentiated
form,
  \begin{equation} \label{SPPROJDAMPED} d_{\lambda} P_{[0, \lambda]
  }^{\tau}(\zeta, \bar{\zeta}) = \sum_j \delta(\lambda -
 \lambda_j) e^{- 2 \tau \lambda_j} |\phi_j^{\C}(\zeta)|^2, \end{equation}
 which is a  temperate distribution on $\R$ for each $\zeta$  satisfying $\sqrt{\rho}(\zeta)
 \leq \tau. $
When we set $\tau = \sqrt{\rho}(\zeta)$ we omit the
 $\tau$ and write
  \begin{equation} \label{SPPROJDAMPEDz} d_{\lambda} P_{[0, \lambda]
  }(\zeta, \bar{\zeta}) = \sum_j \delta(\lambda -
 \lambda_j) e^{- 2 \sqrt{\rho}(\zeta) \lambda_j}
 |\phi_j^{\C}(\zeta)|^2. \end{equation}
The advantage of the tempered projections is that they have
polynomial asymptotics and one can use standard Tauberian theorems
to analyse their growth.

We prove the following  one-term local Weyl law for complexified
spectral projections:
\begin{maintheo}\label{PTAULWL}  On any compact real analytic Riemannian manifold $(M,g)$ of dimension $n$,  we have,  with remainders uniform in
$\zeta$,
\begin{enumerate}

\item For $\sqrt{\rho}(\zeta) \geq \frac{C}{\lambda}, $
$$ P_{[0, \lambda]}(\zeta, \bar{\zeta}) =  (2\pi)^{-n} \left(\frac{\lambda}{\sqrt{\rho}} \right)^{\frac{n-1}{2}}
  \left( \frac{\lambda}{(n-1)/2 + 1} +  O (1) \right);
$$

\item For $\sqrt{\rho}(\zeta) \leq \frac{C}{\lambda}, $ $$ P_{[0,
\lambda]}(\zeta, \bar{\zeta}) =  (2\pi)^{-n} \; \lambda^{n}
\left(1 + O(\lambda^{-1}) \right]. $$
\end{enumerate}

\end{maintheo}

This implies new bounds on pointwise norms on complexified
eigenfunctions, improving those of \cite{GLS}.
inequality gives 

\begin{maincor} \label{PWa} Suppose  $(M, g)$ is real analytic of dimension $n$,  and
that $I_{\lambda} = [0, \lambda]$. Then,

\begin{enumerate}

\item For $\tau \geq \frac{C}{\lambda} $ and $\sqrt{\rho}(\zeta) = \tau$, 
there exists $C > 0$ so that
$$ C
\lambda_j^{-\frac{n-1}{2}} e^{ \tau \lambda} \leq \sup_{\zeta \in
M_{\tau}} |\phi^{\C}_{\lambda}(\zeta)| \leq C
  \lambda^{\frac{n-1}{4} + \half} e^{\tau \lambda}.
$$

\item  For $\tau\leq \frac{C}{\lambda}, $ and $\sqrt{\rho}(\zeta) = \tau$,  there
exists $C > 0$ so that
$$ |\phi^{\C}_{\lambda}(\zeta)| \leq \lambda^{\frac{n - 1}{2}};
$$

\end{enumerate}

\end{maincor}

The lower bound of Corollary \ref{PWa} (1)    combines  Theorem \ref{PTAULWL} with G\"arding's inequality. 
The upper bound  sharpens the estimates claimed in \cite{Bou,GLS},
\begin{equation} \label{SUPESTEIG} \sup_{\zeta \in M_{\tau}} |\phi^{\C}_{\lambda}(\zeta)|
\leq C_{\tau} \lambda^{n + 1} e^{\tau \lambda}.
\end{equation}
The improvement is due to using spectral asymptotics rather than a
crude Sobolev inequality.

The next Proposition ties together the work on  triple inner
products of eigenfunctions in \cite{Sar,BR} and elsewhere  with
analytic continuations of eigenfunctions to Grauert tubes.  The
basic question is the decay rate of the inner products $\int_M
\phi_{\lambda_j} \phi_{\lambda_k}^2 dV_g$ where $dV_g$ is the
volume form of $(M, g)$. More generally, one considers integrals
where $\phi_{\lambda_k}^2$ is replaced by a polynomial in
eigenfunctions of fixed eigenvalues.  In \cite{Sar}, it is proved
that
 $|\langle P,\phi_{\lambda_k}\rangle |\le
A(\lambda_k+1)^B\exp(-\pi\sqrt{\lambda_k}/2)$. The exponent is
sharp, but the prefactor is improved in \cite{BR}. The exponent
constant $\frac{\pi}{2}$ is the  radius of the maximal Grauert
tube for hyperbolic space and its quotients (see \cite{Sz,KM} and
\S \ref{EXAMPLES} for the latter fact). The next Proposition
generalizes this bound to any real analytic metric. The radius
$\tau_{an}$ is the maximal {\it analytic tube radius} defined in
Definition \ref{MAXGRAU}. Essentially, it is the largest tube to
which all eigenfunctions analytically continue. Its relation to the geometric
radius is discussed in \S \ref{MAXG} and \S \ref{MAX}.

\begin{mainprop} \label{TRIPLEPROD} Let $(M, g)$ be any compact
real analytic manifold and let $\tau_{an}(g)$ be the maximal
analytic Grauert tube  radius. Then, for all $\tau < \tau_{an}$,
there exists a constant $C_{\tau}$ such that
$$|\int_M \phi_{\lambda_j} \phi_{\lambda_k}^2 dV_g| \leq
C_{\tau} (\lambda_k) e^{- \tau \lambda_j}. $$ If $\partial
M_{\tau_{an}(g)}$ is a smooth manifold and $\phi_{\lambda_k}^{\C}$
is a distribution of order $r$ on $\partial M_{\tau_{an}(g)}$,
then there exists a constant $C$ so that
$$|\int_M \phi_{\lambda_j} \phi_{\lambda_k}^2 dV_g| \leq
C (\lambda_k) \lambda_j^r  e^{- \tau_{an}(g) \lambda_j}. $$
\end{mainprop}

As will be seen in the proof, $C(\lambda_k)$ is a Sobolev $W_s$
norm of $e^{\tau \sqrt{\Delta}} \phi_{\lambda_k}$.  The statement
lacks the precision of the hyperbolic case, since we do not determine
whether $\partial M_{\tau}$ is even a smooth manifold. In \S \ref{MAX},
we sketch a proof that  $\tau_{an}$ is the usual geometric radius of the
Grauert tube, and then the estimate of Proposition \ref{TRIPLEPROD} has almost the
same exponential asymptotics as in the hyperbolic case.

Finally, we give an application to complex zeros of the joint
eigenfunctions of the algebra $\dcal$ of invariant differential
operators on  the locally symmetric space $SO(n, \R) \backslash
SL_n(\R)/\Gamma$, where $\Gamma $ is a co-compact discrete
subgroup of $SL_n(\R)$ In \cite{AS}, Anantharaman-Silberman proved
a number of results on the entropies of the quantum limit measures
of the joint eigenfunctions as the joint eigenvalue tends to
infinity. Roughly speaking, the results say that the quantum limit
measures must have a non-trivial Haar component. This result is
sufficient to determine the limit distribution of complex zeros of
the complexifications of the same eigenfunctions.
 We denote by $[Z_{\phi_{\lambda}^{\C}}]$ the
current of integration over the complex zero set of
$\phi_{\lambda}^{\C}$. 

\begin{maintheo}\label{ZERORAN}  Let $(M, g)$ be a compact locally symmetric manifold, and let
$\{\phi_{\lambda}\}$ be any orthonormal basis of the joint
$\dcal$-eigenfunctions.  Then for $\tau < \tau_{an}$.
$$\frac{1}{\lambda} [Z_{\phi_{\lambda}^{\C}}] \to  \frac{i}{ \pi} \ddbar \sqrt{\rho},\;\;
 \mbox{weakly in}\;\; \dcal^{' (1,1)} (M_{\tau}), $$
 for the entire sequence of eigenfunctions  $\{\psi_{j}\}$.
\end{maintheo}

This proof requires just a small observation improving on the weak
convergence result of \cite{Z3},   placed on top of the very
strong Haar component theorem of Anantharaman-Silberman.

\subsection{Results of \cite{Z1}}

The asymptotics of the complexified spectral projection kernels
\eqref{CXSPMa} are complex analogues of those of the diagonal
spectral projections in the real domain and reflect the structure
of complex geodesics from $\zeta$ to $\bar{\zeta}$. As in the real
domain, one can obtain  more refined asymptotics of $P_{[\lambda,
\lambda + 1]}(\zeta, \bar{\zeta})$ by using the structure of
geodesic segments from $\zeta$ to $\bar{\zeta}$. This is the
subject of the sequel \cite{Z1}. For the sake of completeness, we
state the results here: There exists an explicit  complex
oscillatory factor $Q_{\zeta}(\lambda)$  depending on the geodesic
arc from $\zeta$ to $\bar{\zeta}$ such that
\begin{enumerate}

\item For $\sqrt{\rho}(\zeta) \geq \frac{C}{\lambda}, $
$$ P^{\tau}_{[0, \lambda] }(\zeta, \bar{\zeta}) =  (2\pi)^{-n} \lambda \left(\frac{\lambda}{\sqrt{\rho}} \right)^{\frac{n - 1}{2}}
  \left( 1 +  Q_{\zeta}(\lambda) \lambda^{-1} +  o(\lambda^{-1})  \right);
$$

\item For $\sqrt{\rho}(\zeta) \leq \frac{C}{\lambda}, $ $$
P^{\tau}_{[0, \lambda]}(\zeta, \bar{\zeta}) =  (2\pi)^{-n} \;
\lambda^{n} +   Q_{\zeta}(\lambda) \lambda^{n-1} + o(\lambda^{n-1}
),
$$
\end{enumerate}
  The functions $Q_{\zeta}(\lambda)$
 depend on whether $(M, g)$ is
a manifold without conjugate points, or with conjugate points. We
refer to \cite{Z1} for the formulae.  A special case is that of
Zoll manifolds where there exists a complete asymptotic expansion
similar to that for line bundles. The two term asymptotics lead to
improvement by one order of magnitude on the bounds in Corollary
\ref{PWa}, and are sharp in that they are achieved by complexified
zonal spherical harmonics on a standard sphere.

  \section{\label{GRAUERT} Grauert tubes and complex geodesic flow}

 By a
theorem of Bruhat-Whitney, a real analytic Riemannian manifold $M$
admits a complexification $M_{\C}$, i.e. a complex manifold into
which $M$ embeds as a totally real submanifold. Corresponding to a
real analytic metric $g$ is a unique plurisubharmonic exhaustion
function $\sqrt{\rho}$ on $M_{\C}$ satisfying two conditions (i)
It satisfies the Monge-Amp\`ere equation $(i \ddbar \sqrt{\rho})^n
= \delta_{M, g}$ where $\delta_{M,g} $ is the delta function on
$M$ with density $dV_g$ equal to the volume density of $g$; (ii)
the \kahler metric $\omega_g = i \ddbar \rho$ on $M_{\C}$ agrees
with $g$ along $M$. In fact,  \begin{equation} \label{rhoeq}
\sqrt{\rho}(\zeta) = \frac{1}{2 i} \sqrt{r^2_{\C}(\zeta,
\bar{\zeta})}, \end{equation} where $r^2(x,y)$ is the square of
the distance function and $r^2_{\C}$ is its holomorphic extension
to a small neighborhood of the anti-diagonal $(\zeta,
\bar{\zeta})$ in $M_{\C} \times M_{\C}$. In the case of flat
$\R^n$, $\sqrt{\rho}(x + i \xi) = 2 |\xi|$ and in general
$\sqrt{\rho}(\zeta)$ measures how far $\zeta$ reaches into the
complexification of $M$.  The open Grauert tube of radius $\tau$
is defined by  $M_{\tau} = \{\zeta \in M_{\C}, \sqrt{\rho}(\zeta)
< \tau\}$. We use the imprecise notation $M_{\C}$ to denote the
open complexificaiton when it is not
  important to specify the radius.

  \subsection{Analytic continuation of the exponential map}

The geodesic flow is a Hamiltonian flow on $T^*M$. In fact, there
are two standard choices of the Hamiltonian.  In PDE it is most
common to define the  (real) homogeneous geodesic flow $g^t$ of
$(M, g)$ as the Hamiltonian flow on $T^*M$ generated by the
Hamiltonian $|\xi|_g$ with respect to the standard Hamiltonian
form $\omega$.  This Hamiltonian is real analytic on $T^*M
\backslash 0$.  In Riemannian geometry it is  standard to let the
time of travel equal $|\xi|_g$; this corresponds to the
Hamiltonian flow of $|\xi|_g^2$, which is real analytic on all of
$T^*M$. We denote its Hamiltonian flow by $G^t$. In general, we
denote by $\Xi_H$ the Hamiltonian vector field of a Hamiltonian
$H$ and its flow by $\exp t \Xi_H$.  Both of the Hamiltonian flows
\begin{itemize}

\item $g^t = \exp t \Xi_{|\xi|_g}$;

\item $G^t = \exp t \Xi_{|\xi|_g^2} $

\end{itemize}
are important in analytic continuation of the wave kernel. The
exponential map is the map $\exp_x: T^*M \to M$ defined by $\exp_x
\xi = \pi G^t(x, \xi)$ where $\pi$ is the standard projection.

We denote by
  $\mbox{inj}(x)$ the injectivity radius of $(M, g)$ at $x$, i.e.
  the radius $r$ of the largest ball  on which $\exp_x: B_r M \to
  M$ is a diffeomorphism to its image.
Since $(M, g)$ is real analytic, $\exp_x t \xi$ admits an analytic
continuation in $t$ and the imaginary time exponential map
\begin{equation} \label{EXP} E: B_{\epsilon}^* M \to M_{\C}, \;\;\; E(x, \xi) =
\exp_x i \xi \end{equation} is, for small enough $\epsilon$, a
diffeomorphism from the ball bundle $B^*_{\epsilon} M$ of radius
$\epsilon $ in $T^*M$ to the Grauert tube $M_{\epsilon}$ in
$M_{\C}$. We have $E^* \omega = \omega_{T^*M}$ where $\omega = i
\ddbar \rho$ and where $\omega_{T^*M}$ is the canonical symplectic
form; and also $E^* \sqrt{\rho} = |\xi|$ \cite{GS1,LS1}. It
follows that $E^*$ conjugates the geodesic flow on $B^*M$ to the
Hamiltonian flow $\exp t \Xi_{\sqrt{\rho}}$ of $\sqrt{\rho}$ with
respect to $\omega$, i.e.
$$E(g^t(x, \xi)) = \exp t \Xi_{\sqrt{\rho}} (\exp_x i \xi). $$

\subsection{\label{MAXG} Maximal Grauert tubes}

A natural definition of {\it maximal Grauert tube} is the maximum
value of $\epsilon$ so that (\ref{EXP}) is a diffeomorphism. We
refer to this radius as the {\it maximal geometric tube radius}.
But for purposes of this paper, another definition of maximality
is relevant: the maximal tube on which all eigenfunctions extend
holomorphically. A closely related definition is the maximal tube
to which the Poisson kernel (\ref{POISSWAVE}) extends
holomorphically. We refer to the radius as the {\it maximal
analytic tube radius}.

A natural question is to relate these  notions of maximal
Grauert tube has not been explored. We therefore define the radii
more precisely:

\begin{defn} \label{MAXGRAU}

\begin{enumerate}
\medskip


\item  The maximal geometric tube radius  $\tau_{g}$  is the
largest radius $\epsilon$ for which $E$ (\ref{EXP}) is a
diffeomorphism.

\item The maximal analytic  tube radius  $\tau_{an}$
$M_{\tau_{an}} \subset M_{\C}$ is the maximal tube to which all
eigenfunctions extend holomorphically and to which the
anti-diagonal  $U(2i \tau, \zeta, \bar{\zeta})$ of the Poisson
kernel admits an analytic continuation.

\end{enumerate}

\end{defn}

We make:

\begin{conj} $\tau_g = \tau_{an}$. \end{conj}

In \S \ref{MAX} we sketch the proof, which is based on holomorphic extensions of solutions of
analytic PDE across non-characteristic hypersurfaces. It would require too much background to
incude a more detailed proof here, but we hope the sketch of proof indicates the main ideas. We found
a similar argument in \cite{KS} in the case of locally symmetric spaces but employing additional arguments.
We intend to give more details in \cite{Z5}.

\subsection{\label{EXAMPLES} Model examples}.

We consider some standard  examples to clarify these analytic
continuations.

\medskip

\noindent{\bf (i) } Complex tori:
\medskip

The complexification of the  torus $M = \R^m/\Z^m$ is $M_{\C} =
\C^m/\Z^m$. The adapted complex structure to the flat metric on
$M$  is the standard (unique)  complex structure on $\C^m$. The
complexified exponential map is $\exp_{ x}^{\C} (i \xi) = z: =  x
+ i \xi $, while  the distance function $r(x, y) = |x - y|$
extends to $r_{\C}(z, w) = \sqrt{(z - w)^2}. $ Then
$\sqrt{\rho}(z, \bar{z}) = \sqrt{(z - \bar{z})^2} = \pm 2 i |\Im
z| = \pm 2 i |\xi|. $

The complexified cotangent bundle is $T^* M_{\C} = \C^m/\Z^m
\times \C^m$, and the holomorphic geodesic flow is the entire
holomorphic map
$$G^t(\zeta, p_{\zeta}) = (\zeta  + t p_{\zeta}, p_{\zeta}). $$

\noindent{\bf  (ii)  $\Ss^n$} \cite{GS1}  The unit sphere $x_1^2 +
\cdots + x_{n+1}^2 = 1$ in $\R^{n+1}$ is complexified as the
complex quadric
$$\Ss^n_{\C} = \{(z_1, \dots, z_n) \in \C^{n + 1}: z_1^2 + \cdots +
z_{n+1}^2 = 1\}. $$ If we write $z_j = x_j + i \xi_j$, the
equations become $|x|^2 - |\xi|^2 = 1, \langle x, \xi \rangle =
0$. The geodesic flow $G^t(x, \xi) = (\cos t |\xi|) x + (\sin t
|\xi|) \frac{\xi}{|\xi|}, - |\xi| (\sin t  |\xi|) x + (\cos t
|\xi|) \xi)$ on $T^* \Ss^n$ complexifies to
$$\begin{array}{lll} G^t(Z, W) && = (\cos t \sqrt{W \cdot W} ) Z + (\sin t
\sqrt{W \cdot W} )) \frac{W}{\sqrt{W \cdot W} )},\\ &&\\
&& - \sqrt{W \cdot W} ) (\sin t \sqrt{W \cdot W} )) Z + (\cos t
\sqrt{W \cdot W} )) W),\;\; ((Z, W) \in T^* \Ss_{\C}^m). \end{array}$$
Here, the real cotangent bundle is the subset of $T^* \R^{n + 1}$
of $(x, \xi)$ such that $x \in \Ss^n, x \cdot \xi = 0$ and the
complexified cotangent bundle $T^* \Ss^n_{\C} \subset T^*
\C^{n+1}$ is the set of vectors $(Z, W): Z \cdot W  = 0$.  We note
that although $\sqrt{W \cdot W}$ is singular at $W = 0$, both
$\cos \sqrt{W \cdot W} )$ and $\sqrt{W \cdot W} ) \sin t \sqrt{W
\cdot W} )$ are holomorphic. The Grauert tube function equals
$$\sqrt{\rho}(z) =  i \cosh^{-1} |z|^2, \;\;(z \in \Ss^n_{\C}). $$
It is globally well defined on $\Ss^n_{\C}$. The characteristic
conoid is defined by $\cosh \frac{1}{i} \sqrt{\rho} = \cosh \tau$.

\medskip

\noindent{\bf (iii) (See e.g. \cite{KM}).  $\Hh^n$} The
hyperboloid model of hyperbolic space is the hypersurface in
$\R^{n+1}$ defined by
$$\Hh^n = \{ x_1^2 + \cdots
x_n^2 - x_{n+1}^2 = -1, \;\; x_n > 0\}. $$ Then,
$$H^n_{\C} = \{(z_1, \dots, z_{n+1}) \in \C^{n+1}:  z_1^2 + \cdots
z_n^2 - z_{n+1}^2 = -1\}. $$ In real coordinates $z_j = x_j + i
\xi_j$, this is:
$$\langle x, x \rangle_L - \langle \xi, \xi \rangle_L = -1,\;\;
\langle x, \xi \rangle_L = 0$$ where $\langle, \rangle_L$ is the
Lorentz inner product of signature $(n, 1)$. Hence the
complexified hyperbolic space is the hypersurface in $\C^{n + 1}$
given by the same equations.

We obtain $\Hh^n_{\C}$ from $\Ss^n_{\C}$ by the map $(z', z_{n +
1}) \to (i z', z_{n + 1})$. The complexified geodesic flow is
given by $$\begin{array}{lll} G^t(Z, W) && = (\cosh t
\sqrt{\langle W,  W \rangle_L}  Z + (\sinh t
\sqrt{\langle W,  W \rangle_L} )) \frac{W}{\sqrt{\langle W,  W \rangle_L} )},\\ &&\\
&& - \sqrt{\langle W,  W \rangle_L} ) (\sinh t \sqrt{\langle W,  W
\rangle_L} )) Z + (\cosh t \sqrt{\langle W,  W \rangle_L} ))
W),\;\; ((Z, W) \in T^* \Hh^m).
\end{array}$$

The Grauert tube  function is:
$$\sqrt{\rho}(z) = \cos^{-1} (||x||_L^2 + ||\xi||_L^2 - \pi)/\sqrt{2}. $$
 The radius of maximal Grauert
tube is $\epsilon = 1$ or $r = \pi/\sqrt{2}.$

\section{$\Delta_g, \Box_g $ and characteristics}

The Laplacian of $(M, g)$ is given in local coordinates by  $$\Delta_g =   
- \frac{1}{\Theta} \sum_{i, j} \frac{\partial}{\partial x_i} (\Theta
g^{i j} ) \frac{\partial}{\partial x_j}, $$ is the Laplacian of
$(M,g)$.  Here, $g_{ij} = g(\frac{\partial}{\partial
x_i},\frac{\partial}{\partial x_j}) $, $[g^{ij}]$ is the inverse
matrix to $[g_{ij}]$ and $\Theta = \sqrt{{\rm det} [g_{ij}]}.$  Since  $g$ is fixed we henceforth
write the Laplacian as $\Delta$.  Note that we have put a minus sign in front of the sum of squares
to make $\Delta$ a non-negative operator. This is for later notational convenience. On
a compact manifold, $\Delta$ is  negative operator with discrete spectrum
\begin{equation} \Delta \phi_j =  \lambda_j^2 \phi_j,\;\;\;\;\;\;\;\;\langle \phi_j, \phi_k
\rangle = \delta_{jk} \end{equation}
 of eigenvalues and eigenfunctions.  Note that the eigenvalues are denoted $\lambda_j^2$; we refer to
$\lambda_j$ as the `frequency'. 

In the real domain, $\Delta$ is an elliptic operator  with principal symbol
$\sigma_{\Delta} (x, \xi) =  :  \sum_{i, j = 1}^n g^{i j}(x) \xi_i \xi_j $. Hence its characteristic set (the zero set of its symbol)
consists only of the zero section $\xi = 0$ in $T^* M$. But when we continue it to the complex domain it develops
a  complex characteristic set
\begin{equation} \label{DELTACHAR} \mbox{Ch}(\Delta_{\C}) = \{ (\zeta, \xi)  \in T^* M_{\C} : \sum_{i, j = 1}^n g^{i j}(\zeta) \xi_i \xi_j = 0 \}. 
\end{equation}

The wave operator on the product spacetime  $(\R \times M, dt^2 - g_x)$ is given by
$$\Box_g = \frac{\partial^2}{\partial t^2} + \Delta_g. $$
The unusual sign in front of $\Delta_g$ is due to the sign normalization above making the Laplacian non-negative. 
Again we omit the subscript when the metric is fixed.
The characteristic variety of $\Box$ is the zero set of its symbol
$$\sigma_{\Box} (t, \tau, x, \xi) =   \tau^2 - |\xi|_x^2, $$
that is,
\begin{equation} \label{BOXCHAR} \mbox{Ch}(\Box) = \{(t, \tau, x, \xi) \in T^* (\R \times M):  \tau^2 - |\xi|_x^2 = 0\}.  \end{equation} 

The null-bicharacteristic flow of $\Box$ is the Hamiltonian flow of $\tau^2 - |\xi|_x^2$ on
Ch($\Box$). Its graph is thus
$$\Lambda= \{(t, \tau, x, \xi, y, \eta): \tau^2 - |\xi|_x^2 = 0, G^t(x, \xi) = (y, \eta) \} \subset T^*(\R \times M \times M). $$

\subsection{Characteristic variety and characteristic conoid}

Following \cite{H}, we put
\begin{equation} \label{GAMMA} \Gamma(t, x, y) = t^2 - r^2(x, y). \end{equation}
Here, $r(x, y)$ is the  distance  between $ x, y$.  It  is singular at $r = 0$ and also  when y is in the ``cut locus'' of $x$.
In this article we only consider $(x, y)$ so that $r(x, y) < \mbox{inj}(x)$, where inj(x) is 
the injectivity radius at $x$, i.e.  is the largest $\epsilon$ so that
$$\exp_x: B^*_{x, \epsilon} M \to M$$
is a diffeomorphism to its image.  The injectivity radius $inj(M, g)$ is the maximum of $\mbox{inj}(x)$ for $x \in M$.
Thus, we work  in a sufficiently small neighborhood of the diagonal so that cut points do not occur.

The squared distance  $r^2(x, y)$ is smooth in a neighborhoof of the diagonal. On a simply
connected manifold $(\tilde{M}, g)$ without conjugate points, it is globally smooth on
$\tilde{M} \times \tilde{M}$. We recall that `without conjugate points' means that $\exp_x : T_x M \to M$ 
is non-singular for all $x$.

The characteristic conoid is the set
\begin{equation} \label{REALCHCON} \ccal = \{(t, x, y): r(x,
y) < \mbox{inj}(x), \;\; r^2(x, y) = t^2\} \subset \R \times M
\times M.
\end{equation} 
It separates $\R \times M \times M$  into  the forward/backward semi-cones
$$\ccal_{\pm} =  \{(t, x, y): t^2 - r^2(x, y)  >  0, \pm t > 0\} . $$

 The  complexificationof $\ccal = \ccal_{\R}$  is the complex
characteristic conoid
\begin{equation} \label{CXCHCON} \ccal_{\C} = \{(t, x, y):
r_{\C}^2(x, y) = t^2\} \subset \C \times M_{\C} \times M_{\C}.
\end{equation}
We note that $\ccal_{\R} \subset \ccal_{\C}$ is a totally real
submanifold. Another totally real submanifold of central
importance in this article is the `diagonal' (or anti-diagonal)
conoid,
\begin{equation} \label{DCHCON} \ccal_{\Delta} = \{(2 i \tau, \zeta,
\bar{\zeta}): \tau \in \R_+, \zeta, \bar{\zeta} \in \partial
M_{\tau} \}.
\end{equation}
By definition, $r^2_{\C}(\zeta, \bar{\zeta}) = - 4 \tau^2$ if
$\zeta \in \partial M_{\tau}$.

\section{Propagators and fundamental solutions}

The main `wave kernels' in this article are the half-wave kernel $e^{i t \sqrt{\Delta}}$ and the 
Poisson kernel $e^{- \tau \sqrt{\Delta}}$ for $\tau > 0$. To put these kernels into context, we now
give a brief review of propagators and fundamental solutions for the wave equation.  We use the
term `propagator' for a solution operator to a Cauchy problem. It will be a homogeneous solution
of $\Box E = 0$ with special initial conditions. We use the term `fundamental solution' for a solution
of the inhomogeneous equation $\Box E = \delta_0$. We freely use  standard notation for homogeneous
distributions on $\R$ and refer to \cite{Ho} for notation and background.

\subsection{Cauchy problem for the wave equation}

The Cauchy problem for the wave equation on $\R \times M$ is the initial value problem
(with Cauchy data $f, g$ ) 
$$\left\{ \begin{array}{l} \Box u(t, x) = 0, \\ \\
u(0, x) = f, \;\; \frac{\partial}{\partial t} u(0, x) = g(x), \end{array} \right.. $$

The solution operator of the Cauchy problem (the ``propagator") is the wave group,
$$\ucal(t) = \begin{pmatrix} \cos t \sqrt{\Delta} & \frac{\sin t \sqrt{\Delta}}{\sqrt{\Delta}} \\ & \\
\sqrt{\Delta} \sin t \sqrt{\Delta}  & \cos t \sqrt{\Delta} \end{pmatrix}. $$
The solution of the Cauchy problem with data $(f, g)$ is $\ucal(t) \begin{pmatrix} f  \\ g \end{pmatrix}. $

Two of the components of $\ucal(t)$ are particularly important:

\begin{itemize}

\item The even part
$\cos t\sqrt{\Delta}$ is the solution operator of the initial value
problem,
\begin{equation} \left\{ \begin{array}{ll} \Box  u = 0& \\
u|_{t=0} = f & \frac{\partial}{\partial t} u |_{t=0} = 0
\end{array}\right .\end{equation} 

\item  The odd part $\frac{\sin t\sqrt{\Delta}}{\sqrt{\Delta}}$ is the
solution operator of the initial value problem,
\begin{equation} \left\{ \begin{array}{ll} \Box u = 0& \\
u|_{t=0} = 0 & \frac{\partial}{\partial t} u |_{t=0} = g
\end{array}\right .\end{equation} 

\end{itemize}

The kernels of $\cos t \sqrt{\Delta}, \frac{\sin t
\sqrt{\Delta}}{\sqrt{\Delta}}$ exhibit finite propagation speed of
solutions of the wave equation,  i.e.  are supported inside the characteristic
conoid $\ccal$ where  $r \leq |t|.$ 

\subsection{Cauchy problem for the half-wave equation}

The forward  half-wave group is the solution operator of the Cauchy problem
$$(\frac{1}{i } \frac{\partial}{\partial t} - \sqrt{\Delta} ) u = 0, \;\;\;  u(0, x) = u_0. $$
The solution is given by 
$$u(t, x) = U(t) u_0(x), \;\;\;
\mbox{with}\;\;U(t) = e^{i t \sqrt{\Delta}}.$$

The Schwartz kernel $U(t, x, y)$ of the wave group  $U(t) = e^{ i
 t \sqrt{\Delta}}$ solves the pseudo-differential Cauchy
 problem(the half-wave equation),
\begin{equation} \label{HALFWE} \left(\frac{1}{i} \frac{\partial }{\partial t} -
\sqrt{\Delta}_x \right) U(t, x, y) = 0, \;\; U(0, x, y) =
\delta_y(x). \end{equation}   Equivalently, it solves the wave
equation with pseudo-differential initial condition,
\begin{equation} \label{WE}\left\{ \begin{array}{ll}
\Box \; U  = 0, & \\
\\ \;\; U(0, x, y) = \delta_y(x),  & \frac{\partial }{\partial t}
U(t, x, y)|_{t = 0} = i \sqrt{\Delta}_x \; \delta_x(y). \end{array}
\right. \end{equation} The solution is given by
\begin{equation} \label{SOL} U(t, x, y) = \cos t \sqrt{\Delta}(x, y) + i \sqrt{\Delta}_x \; \frac{\sin t
\sqrt{\Delta}}{\sqrt{\Delta}}(x, y). \end{equation} 
Unlike the even/odd kernels, $e^{i t \sqrt{\Delta}}$
 has infinite propagation speed,  i.e. is non-zero outside the characteristic conoid $\ccal$; this
is due to the second of its initial condition.

 The half wave group   has the eigenfunction expansion,
 \begin{equation} \label{efnex} U (t, x, y) = \sum_j
e^{i t \lambda_j} \phi_{\lambda_j}(x) \phi_{\lambda_j}(y)
\end{equation}
on $\R \times M \times M$,
which converges in the sense of distributions.

\subsection{Fundamental solutions}

A fundamental solution of the wave equation is a solution of 
$$\Box E(t, x, y) = \delta_0(t) \delta_x(y). $$
The right side is the Schwartz kernel of the identity operator on $\R \times M$. 

There exists a unique fundamental solution which is supported in the forward conoid
$$C_+ = \{(t, x, y): t > 0,  t^2 - r^2(x, y) > 0\}. $$ called
the advanced (or forward) propagator. It is given by
$$E_+(t) = H(t) \frac{\sin t \sqrt{\Delta}}{\sqrt{\Delta}}, $$
where $H(t) = {\bf 1}_{t \geq 0}$ is the Heaviside step function. 
It is well-defined for any curved globally hyperbolic spacetime, while Cauchy problems and propagators
require a choice of ``Cauchy hypersurface" like $\{t = 0\}$.

 In \cite{R} , the forward fundamental solution  is constructed
in terms of the holomorphic family of  Riesz kernels $\frac{(t^2 - r^2)_+^{\alpha}}{\Gamma(\alpha + 1)}$,
which are supported in the forward characteristic conoid  $\ccal_+$.  A more contemporary treatment using the language of homogeneous
distributions on $\R$ is given in \cite{Be}. In \cite{J} it is pointed out that the Riesz kernels are
Schwartz kernels of complex powers $\Box^{\alpha}$ of the wave operator on $\R \times M$. Unlike
complex powers of $\Delta$,  $\Box^{\alpha}$ is only uniquely defined if the Scwhartz kernels are assumed
to be supported in $\ccal_+$.

\subsection{Hadamard-Feynman fundamental solution}

Hadamard and Feynman constructed another fundamental solution which is a (branched)
meromorphic  function of $(t, x, y)$ near the characteristic conoid  with the singularity $(t^2 - r^2)^{-\frac{m-2}{2}}$ analogous to the Newtonian potential
$r^{2 - n}$ in the elliptic case (here $m = n + 1 = \dim \R \times M^n$.)
 and is not supported in $\ccal_+$. It corresponds to the
inverse $(\Box + i 0)^{-1}$ rather than to the Riesz kernel $\Box^{-1}$.  It is the

 For background in the 
case of $\R^n \times \R$ we refer to \cite{IZ} and for the general case we refer to \cite{DH}. 
Hadamard \cite{H} defined this fundamental solution  to be the branched meromorphic fundamental solution of $\Box$,
and referred to it as the `elementary solution. 
 We review  his parametrix
extensively in \S \ref{HF}.

\begin{defn} The Hadamard-Feynman fundamental solution is the operator $$ (\Box + i 0)^{-1}
= \int_{\R} e^{i t \tau} (\Delta - \tau^2 + i 0)^{-1} dt $$ on $\R \times M$.
\end{defn}

\begin{prop} \label{UF}  As a family $U_F(t)$ of operators on $L^2(M)$ it is given by 
$$U_F(t) = \frac{e^{i |t| \sqrt{\Delta}}}{\sqrt{\Delta}}. $$
\end{prop}

\begin{proof} The proof is essentially the same as in the case of $M = \R^n$ (see for instance \cite{IZ}). 
Using the eigenfunction expansion  $$(\Delta - \tau^2 + i 0)^{-1} = \sum_j (\lambda^2 - \tau^2 + i 0)^{-1}
\phi_j(x) \phi_j(y) $$
it suffices to show that 
$$\int_{\R} e^{i t \tau}   (\lambda_j^2 - \tau^2 + i 0)^{-1} d\tau = \frac{e^{i |t| \lambda_j}}{\lambda_j}. $$
The evaluation follows by a residue calculation.

\end{proof}

\begin{rem} One can verify that $U_F(t)$ is a fundamental solution directly by applying $\Box$. \end{rem}

The analogous expressions for  the advanced (resp. retarded) Green's function are given by
$$G_{ret} (t, x, y) = - \sum_j  \left( \int_{\R} e^{- i t \tau}  ((\tau + i \epsilon)^2 - \lambda_j^2)^{-1} d \tau\;\; \right)\phi_j(x) \phi_j(y). $$

Since
$$ ((\tau + i \epsilon)^2 - \lambda_j^2)^{-1} = \frac{1}{2 \lambda_j} \left( \frac{1}{\tau - \lambda_j + i \epsilon} -  \frac{1}{\tau + \lambda_j + i \epsilon}  \right) $$
we have
$$G_{ret}(t,x, y) =   C_n H(t) \sum_j \frac{\sin t \lambda_j}{\lambda_j} \phi_j(x) \phi_j(y). $$

\subsection{Fundamental solutions and half-wave propagator on $\R^n$}

We illustrate the definitions in the case of $\R^n$ following \cite{Ho}.  We use the notation
 $\chi_+^{\alpha}(x) = \frac{x_+^{\alpha}}{\Gamma(\alpha + 1)}. $
The advanced/retarded fundamental solutions of $\Box$ on $\R^{n + 1} = \R_t \times \R^n_x$ 
is given by
$$ E_{\pm} (t, x, y) =  \chi_{\pm}^{\frac{2-(n + 1)}{2}} (\Gamma),$$
where we use the Hadamard  notation \eqref{GAMMA} (which unfortunately clashes with the Gamma function).

The  Hadamard-Feynamn fundamental solution on $\R^{n  + 1}$  is a ramified (branched) holomorphic
fundamental solution 
\begin{equation} \label{UFRN} U_F(t, x, y) =  (\Gamma + i 0)^{\frac{2-(n + 1)}{2}} . \end{equation}
There is an associated fundamental solution corresponding to $(\Gamma - i 0)^{\frac{2 - (n + 1)}{2}}. $

The half-wave propagator is  constructed on $\R^n$  by the Fourier inversion formula,
\begin{equation} \label{FOURIERINV} U(t, x, y) = \int_{\R^n} e^{i \langle x - y, \xi \rangle} e^{i t |\xi|} d \xi.  \end{equation}
The Poisson kernel (extending functions on $\R^n$ to harmonic functions on $\R_+ \times \R^n$)
is the half-wave propagagor at positive  imaginary times $t = i \tau$ ($\tau >0$),
\begin{equation} \label{UFORMULA} \begin{array}{lll} U(i \tau, x, y) & = & \int_{\R^n} e^{i \langle x - y, \xi \rangle} e^{- \tau |\xi|} d \xi \\ && \\
& = & 
 \tau^{-n} \left(1 + (\frac{x -
y}{\tau})^2 \right)^{-\frac{n+ 1}{2}}  = \tau \left(\tau^2 + (x -
y)^2)\right)^{-\frac{n + 1}{2}}. \end{array} \end{equation}

In the case of $\R^n$, the Poisson kernel analytically continues  to $t + i \tau, \zeta =
x + i p \in \C_+ \times \C^n$ as the integral
\begin{equation} \label{PWAC} U(t + i \tau , x + i p , y) = \int_{\R^n} e^{i (t + i \tau)  |\xi|} e^{i \langle \xi, x + i p - y
\rangle} d\xi,  \end{equation}  which converges absolutely for $|p| < \tau.$ 
If we  substitute $\tau \to \tau - it$ and let $\tau \to 0$ we get the formula 
\begin{equation} \label{UFORM} U(t, x, y) =C_n\;\lim_{\tau \to 0}  i t ( (t + i \tau))^2 -  r(x, y)^2  )^{- \frac{n + 1}{2}}, \end{equation}
for a constant $C_n$ depending only on the dimension. 

\begin{rem} \label{KERNELS}

We observe that the half-wave kernel differs from the Hadamard-Feynman fundamental solution
not only in the power but also because the former uses powers of  the quadratic form $((t + i 0)^2 - r^2)$
while the latter uses $(t^2 - r^2 + i 0)$.  That is, with $r^2 = |x - y|^2$,

\begin{itemize}

\item $C_n \;(t^2 - r^2 + i 0)^{- \frac{n-1}{2}}$ is the Feynman  fundamental solution.\bigskip \bigskip

\item  $C_n'\;  t ( (t + i 0)^2 - r^2)^{- \frac{n +1}{2}}$  is the solution operator kernel of the half-wave equation.

\end{itemize}
\bigskip

This difference in the kernels holds for general $(M, g)$. In the half-wave kernel, $  (t + i \epsilon)^2 = 
 t^2 - \epsilon^2 + 2 i t \epsilon$ and the imaginary
part only has a fixed sign if we assume that $t > 0$.  This in part explains the $|t|$-dependence in 
the formula of Proposition \ref{UF}.

We observe that neither kernel has `finite propagation speed', i.e. neither is supported in the characteristic
conoid.  

\end{rem}

\subsection{\label{SUB} Subordination of the Poisson kernel to the heat kernel}

There is another standard approach to the Poisson kernel based on the  `subordination identity'
$$\begin{array}{lll}  e^{- \gamma} & = & \frac{1}{\sqrt{\pi}} \int_0^{\infty}
\frac{e^{-u}}{\sqrt{u}}   e^{- \frac{\gamma^2}{4 u}} du. 
\end{array}$$
More generally, for any positive operator $A$,
$$e^{- t A} = \frac{t}{2 \sqrt{\pi}} \int_0^{\infty} e^{- \frac{t^2}{4 u}} e^{- u A^2} u^{-3/2} du. $$
This recuces the construction of the Poisson kernel to the heat kernel,  which is useful
since there exists a well-known parametrix for the heat kernel (Levi, Minakshisundaram-Pleijel).
We follow the exposition in \cite{St} \S III.2.

The subordination identity follows from two further  identities:
$$(i)  \; \int_{\R^n} e^{- \pi \delta |t|^2} e^{- 2 \pi i \langle t, x \rangle} dt = \delta^{-n/2} e^{- \pi \frac{|x|^2
}{\delta}}, $$
and
$$(ii) \; e^{- \gamma} = \frac{1}{\sqrt{\pi}} \int_0^{\infty} \frac{e^{-u}}{\sqrt{u}} e^{- \frac{\gamma^2}{4 u}} du. $$
Proof of (ii): We have
$$e^{- \gamma} =  \frac{1}{\sqrt{\pi}} \int_{-\infty}^{\infty} \frac{e^{i  \gamma \cdot x}}{1 + x^2} dx
= \frac{1}{\pi} \int_0^{\infty} e^{-u} \int_{\R} e^{i \gamma x} e^{- u x^2} dx. $$

The formula \eqref{UFORMULA}  for the Poisson kernel
can be obtained from the subordination identity,
$$e^{- \tau |\xi|} =  \frac{1}{\sqrt{\pi}} \int_0^{\infty} \frac{e^{-u}}{\sqrt{u}} e^{- \frac{(\tau |\xi|)^2}{4 u}} du, $$
giving
$$U(i \tau, x, y) =  \frac{1}{\sqrt{\pi}} \int_0^{\infty}  \int_{\R^n}  
 \frac{e^{-u}}{\sqrt{u}} e^{- \frac{(\tau |\xi|)^2}{4 u}}  e^{i \langle x - y, \xi  \rangle} d \xi  du.$$
Interchanging the order of integration then gives,
$$U(i \tau, x, y) =  \frac{1}{\sqrt{\pi}} \int_0^{\infty}   \frac{e^{-u}}{\sqrt{u}}  \left( \int_{T_x M}  
e^{- \frac{(\tau |\xi|)^2}{4 u}}  e^{i \langle x - y, \xi  \rangle}  d \xi \right) du.$$ 

Substituting the formula for the heat kernel
on $\R^n$,
$$e^{- t \Delta} = (4 \pi t)^{-n/2} e^{- \frac{|x - y|^2}{4 t}}, $$
we get 
$$\begin{array}{lll} e^{- t \sqrt{\Delta} } & = &  \frac{t}{2 \sqrt{\pi}} \int_0^{\infty} e^{- \frac{t^2}{4 u}} e^{- u \Delta} u^{-3/2} du \\ && \\ & = &  \frac{t}{2 \sqrt{\pi}} \int_0^{\infty} e^{- \frac{t^2}{4 u}}  (4 \pi u)^{-n/2} e^{- \frac{|x - y|^2}{4 u}}  u^{-3/2} du. \\ && \\ & = &  C_n \tau
\int_0^{\infty} e^{- \theta( \tau^2+ |x - y|^2)}  \theta^{(n - 1)/2}   d \theta
 = C_n \tau (\tau^2 + |x - y|^2)^{-\frac{n + 1}{2}}.\end{array}$$
In the last step we put $\theta = \frac{1}{4 u}.$

\subsection{\label{CONSTCURV} Wave kernels and Poisson kernels on spaces of constant curvature}

As further illustrations, we  consider the Poisson-wave kernels on spaces of non-zero
constant curvature following  \cite{T}.

\subsubsection{The sphere $\Ss^n$}

Let $A = \sqrt{\Delta + (\frac{n-1}{4})^2}. $ Then the Poisson
operator $e^{- t A}$  is given by
$$\begin{array}{lll} U(i \tau, \omega, \omega') &= &  C_n \frac{\sinh t }{(\cosh \tau - \cos
r(\omega, \omega'))^{\frac{n+1}{2}}} \\ & & \\
& = &  C_n \frac{\partial}{\partial \tau}  \left(\cosh \tau - \cos
r(\omega, \omega')\right)^{-\frac{n-1}{2}}. \end{array}$$ Here,
$r(\omega, \omega')$ is the distance between points of $\Ss^n$.
This formula is proved in \cite{T} using the Poisson integral
formula for a ball. Note that the addition of $\frac{n-1}{4}$
simplifies the formula and makes the operator strictly positive.

The wave kernel is the   boundary value:

$$\begin{array}{l} e^{i t A}(x,y)= U(t, x, y)
=  C_n \sin t \left(\cos (t + i 0)  - \cos r(\omega,
\omega'\right)^{-\frac{n+1}{2}} .
\end{array} $$
We  see the kernel structure emphasized  in Remark \ref{KERNELS}.

The analytically continued Poisson kernel is
$$\begin{array}{l} e^{(- \tau + i t) A} (\zeta, \omega) =
 C_n   \sinh t (
\cosh (\tau + i t) -  \cos r(\omega, \zeta))^{-\frac{n+1}{2}}.
\end{array}
$$
 It is singular on the complex
characteristic conoid    $$\cosh (\tau + i t) - \cosh r(\zeta,
\bar{\zeta}') = 0. $$

\subsubsection{Wave kernel on Hyperbolic space}

On hyperbolic space we define $A = \sqrt{\Delta -
(\frac{n-1}{2})^2}$, which brings the continuous spectrum down to
zero. The wave kernel on hyperbolic space is obtained \cite{T}  by
analytic continuation of the wave kernel of $\frac{\sin t A}{A}$
on the sphere by changing $r \to i r$:
$$\lim_{\epsilon \to 0^+} - 2 C_n \Im (\cos (it - \epsilon) -
\cosh r)_+^{-\frac{n-1}{2}}
 $$
 It follows that the Poisson kernel is
$$U(i \tau, x, y) = \sinh t \left(\cosh (t + i 0)- \cosh
r \right)^{-\frac{n+1}{2}}.
$$
We again  see the kernel structure emphasized  in Remark \ref{KERNELS}. The Poisson kernel and wave kernels for$
e^{it \sqrt{\Delta}}$  rather than $e^{i t A}$ are derived in \cite{JL} for hyperbolic
quotients  using the subordination method (\S \ref{SUB}) and the heat kernel. The use of $\sqrt{\Delta}$
rather than $A$ leads to lower order terms. 

There is an alternative approach using Fourier analysis on
hyperbolic space, where the exponential functions $e^{(i
 \lambda + 1) \langle z, b \rangle} $ play the role of plane waves. Here, $z$ is in the interior of hyperbolic
 space and $b$ lies on its boundary and $\langle z, b \rangle$ is the distance of the horosphere through $z, b$
 to $0$. For background
we refer to \cite{Hel}. We have,
 $$e^{i t A}(x, y) = \int_0^{\infty}  e^{i t \lambda} \left\{\int_B e^{(i
 \lambda + 1) \langle z, b \rangle}  e^{(-i
 \lambda + 1) \langle w, b \rangle} db\right\} dp(\lambda), $$
 where $dp(\lambda)$ is the Planchere measure and $db$ is the
 standard measure on the boundary (a sphere). This formula is the
 analogue of  \eqref{FOURIERINV}.  The inner integral over $B$ is a
 spherical function $\phi_{\lambda}(r(z,w))$ and is the hyperbolic
 analogue of a Bessel function. The analytic continuation of the
 Poisson kernel
 $$e^{- \tau A}(x, y) = \int_0^{\infty}  e^{-\tau \lambda} \left\{\int_B e^{(i
 \lambda + 1) \langle z, b \rangle}  e^{(-i
 \lambda + 1) \langle w, b \rangle} db\right\} dp(\lambda), $$ can be easily read off from this expression.

\section{\label{HF} The Hadamard-Feynman fundamental solution and Hadamard's parametrix}

In his seminal work \cite{H},
Hadamard constructed  a solution of $\Box E = 0$ for $t > 0$ which has the
singularity $\Gamma^{\frac{-m + 2}{2}},\;\;\; m = n + 1 = \dim M \times \R $ (recall that
$\Gamma$ is defined by \eqref{GAMMA}).
Note the analogy to the elliptic case where the Green's function (the kernel of $\Delta^{-1}$)
has the singularity $r^{-n + 2}$ if $n > 2$. 

The fundamental
solution is more complicated in  even spacetime dimensions (i.e. odd space dimensions). Hadamard found
the general solution as follows:

\begin{itemize} 

\item The elementary solution in odd spacetime dimensions has the form 
$$U \Gamma^{- \frac{m-2}{2}}, $$ where
$$U= U_0 + \Gamma U_1 + \cdots + \Gamma^h U_h + \cdots$$
is a holomorphic function.  (This $U$ is not the half-wave propagator!).

\bigskip

\item The elementary solution in even spacetime dimensions has the form

$$U \Gamma^{- \frac{m-2}{2}}  + V \log \Gamma + W,$$
where 
$$U = \sum_{j = 0}^{m-1} U_j \Gamma^j, \;\;\; V = \sum_{j = 0}^{\infty} V_j \Gamma^j, \;\; W = \sum_{j = 1}^{\infty}
W_j \Gamma^j. $$

\end{itemize}

Hadamard's formulae for the fundamental solutions pre-date the Schwartz theory of distributions. We follow his
approach of describing the fundamental solutions as branched meromorphic functions (possibly logarithmically
branched) on complexified spacetime. In modern terms $\Gamma^{\alpha}$ (resp. $\log \Gamma$)
would be defined as the distributions $(\Gamma + i 0)^{\alpha}$ (resp. $\log (\Gamma + i 0))$ as in the 
constant curvature cases. Hadamard implicitly worked in the complexified setting.  For background on $\log (x + i 0)$,
see \cite{GeSh} Ch. III.4.4.

\begin{theo} \label{HADAMARD}  (Hadamard, 1920) With the $U_j, V_k, W_{\ell}$ defined as above, 
\begin{itemize}

\item In odd spacetime dimensions, there exists a formal series $U$ as above
so that  $E = U \Gamma^{\frac{2 - m}{2}}$ solves $\Box E = \delta_0(t) \delta_y(x)$.
If $(M, g)$ is real analytic, the  series $U =\sum_{j = 0}^{\infty} U_j \Gamma^j$ converges absolutely for $|\Gamma | < \epsilon$ sufficient
small, i.e. near the characteristic conoid and  admits a  holomorphic continuation to a complex neighborhood of $\ccal_{\C}$.

\item In even spacetime dimensions,  $E = U \Gamma^{\frac{2 - m}{2}} + V \log \Gamma +W $ solves $\Box E = \delta_0(t) \delta_y(x)$. If $(M,g)$ is real analytic, all  of the series for $U, V, W$ converge for $|\Gamma|$ small enough and admit
analytic continuations to a neighborhood of $\ccal_{\C}$

\end{itemize}

\end{theo}

 In the smooth case, the series do not converge.  But if they are truncated  at some $j_0$,  the partial sum defines
a  parametrix, i.e. a fundamental solution modulo functioins in $C^{j_0}$.  By the Levi sums method
(Duhamel principle) the parametrix differs from a true fundamental solution by a $C^{j_0}$ kernels.
We are mainly interested in real analytic $(M,g)$ in this article and do not go into details on the last point. 
We note that the singularities of the kernel are due to the factors $\Gamma^{\frac{2 - m}{2}}, \log \Gamma$,
which are branched meromorphic (and logarithmic) kernels.  The terms are explicitly evaluated in the case of 
hyperbolic quotients in \cite{JL}. We also refer to Chapter 5.2 of \cite{Gar} for a somewhat  modern
presentation of the proof.

It may be of interest to  note that this construction only occupies a third of Hadamard's book \cite{H}. The rest is devoted
to the use of such kernels to solve the Cauchy problem, using Green's formula applied to a domain obtained
by intersecting the backward characteristic conoid from a point $(t, x)$ of spacetime with the Cauchy hypersurface.
The integrals over the lightlike (null) part of the boundary caused serious trouble since the factors $\Gamma^{\frac{2-m}{2}}$
are infinite along them and need to be re-normalized. This was the origin of Hadamard's finite parts of
divergent integrals. Riesz used analytic continuation methods instead  to define the forward
fundamental solution in \cite{R}.

\subsection{Sketch of proof of Hadamard's construction}

Let  $\Theta = \sqrt{\det(g_{jk})}$ be the volume density in
normal coordinates based at $y$, $dV = \Theta(y, x) dx$. That is,
$$\Theta(x, y) = \left| \det D_{\exp_x^{-1}(y)} \exp_x \right|.$$
Fix $x \in M$ and endow $B_{\epsilon}(x)$ with
geodesic polar coordinates $r, \theta$. That is, use the chart $\exp_x^{-1}: B_r(x) \to B^*_{x, r} M$
combined with polar coordinates on $T^*_x M$. 
 Then $g^{11} =
1, g^{1 j} = 0$ for $j =2, \dots, n$.  Also, $dV = \Theta(x, y) dy
= \Theta(x, r, \theta) r^{n-1} dr d \theta$. So the volume density
$J$  relative to Lebesgue measure $dr d \theta$ in polar
coordinates  is given by $J = r^{n - 1} \Theta$.

 In these
coordinates,
$$\Delta = \frac{1}{J} \sum_{j, k =1}^n \frac{\partial
}{\partial x_j} \left( J  g^{jk} \frac{\partial }{\partial x_k}.
\right) = \frac{\partial^2}{\partial r^2} + \frac{J'}{J}
\frac{\partial}{\partial r} + L, $$ where $L$ involves no
$\frac{\partial}{\partial r}$ derivatives. Equivalently,
$$\Delta = \frac{\partial^2}{\partial r^2} +
(\frac{\Theta'}{\Theta}+ \frac{n-1}{r}) \frac{\partial}{\partial
r} + L, $$

The first step in the parametrix construction is to find the phase function. Hadamard chooses
to use $\Gamma$ \eqref{GAMMA}. In the Lortenzian metric, $\Gamma $ satisfies
\begin{equation} \label{GAMMADOT} \nabla \Gamma \cdot \nabla \Gamma = 4 \Gamma. \end{equation}
\bigskip
This is not the standard Eikonal equation $\sigma_{\Box} (d \phi) = 0$ of geometric optics, but 
rather has the form
$$\sigma_{\Box} (d \Gamma) = 4 \Gamma. $$
But $\Gamma$ is a good phase, since the Lagrangian submanifold
$$\{(t, d_t \Gamma, x, d_x \Gamma, y, - d_y \Gamma)\}$$
is the graph of the bichafacteristic  flow. This is because the $d_x r(x, y)$ is the unit vector pointing along
the geodesic joining $x$ to $y$ and $d_y r(x, y)$ is the unit vector pointing along the geodesic
pointing from $y$ to $x$.

To proceed, we introduce the simplifying notation 
$$M = \Box \Gamma = - 4 - 2 r \frac{(n-1) }{r}- 2 r \frac{\Theta_r}{\Theta}  =  2m  + 2 r \frac{\Theta_r}{\Theta}  $$
where $m = n + 1$.  We then have,
$$\begin{array}{lll} \Box \left[f(\Gamma)  U_j \right]  
& =  &\Box  \left[ f(\Gamma) \right] U_j  + 2 \nabla  \left[ f(\Gamma) \right] \nabla U_j +  f(\Gamma) 
\Box U_j
\\ && \\ & = &\left( f''(\Gamma) \nabla (\Gamma) \cdot \nabla (\Gamma) + f'(\Gamma) \Box(\Gamma) \right) U_j + 2 f'(\Gamma) \nabla \Gamma \cdot \nabla U_j + f(\Gamma) \Box U_j. \end{array}$$

In addition to \eqref{GAMMADOT}, we  further have
$$\left\{ \begin{array}{l} \Box \Gamma  =  4  + \frac{J_r}{J}  2 r \\ \\ 
\nabla \Gamma \cdot \nabla = \nabla (t^2 - r^2) \cdot \nabla 
=  2( t \frac{\partial}{\partial t} + r \frac{\partial}{\partial r})  = 2 s \frac{d}{ds}, \end{array} \right.$$
where we recall that  that we are using
the Lorentz metric of signature $+ ---$.   Here $s^2 = \Gamma$, and the notation  $s \frac{d}{ds}$ refers to differentiation along a spacetime geodesic.

We then have
$$\begin{array}{lll} \Box \left[ f(\Gamma) U_j \right] & = &
\left( f''(\Gamma) (4 \Gamma) + f'(\Gamma)  ( 4  + \frac{J_r}{J}  2 r) )\right) U_j
+ 2 f'(\Gamma) (- 2 s \frac{d}{ds} U_j) + f(\Gamma) \Box U_j. \end{array}$$

We now  apply this equation  in the cases   $f = x^{\frac{2 - m}{2} + j}.$
(and later to $f = \log x$), in which case
$$f' = (\frac{2-m}{2}+ j) x^{ \frac{2-m}{2}+ j -1} , \;\; f'' =   (\frac{2-m}{2} + j)   (\frac{2-m}{2}+ j-1) x^{ \frac{2-m}{2}+ j -2}.$$

We then attempt to solve \begin{equation}
\label{FIRST} \Box \left( \Gamma^{\frac{2-m}{2} }\sum_{j = 0}^{\infty} U_j \Gamma^j \right) = 0 \end{equation}
away from the characteristic conoid by seting the coefficient of each power  $\Gamma^{\frac{2 - m}{2} + j - 1}$   of $\Gamma$ equal to zero. The resulting `transport equations' are
$$\begin{array}{l} [ - 4\left(  (\frac{2-m}{2} + j)   (\frac{2-m}{2} + j-1)  \right) 
+(\frac{2-m}{2} + j)  (- 4  - \frac{J_r}{J}  2 r) ) \\ \\
+ 2 \left((\frac{2-m}{2}+ j) (- 2 s \frac{d}{ds}] \right)U_j 
+\Box U_{j-1} = 0. \end{array}$$
They are  impossible to solve for all $j$ when $m$ is even because  the common factor
 $ (\frac{2-m}{2} + j)  $ vanishes when $j = \frac{m-2}{2}$. We thus first assume that $m$ is odd so that it is non-zero
for all $j$ We then recursively solve 
Hadamard's transport equations in even space dimensions,
$$  4s \frac{d U_k}{ds}  + (M - 2m +  2 r \frac{J_r}{J}) U_k = - \Box U_{k-1}. $$
When $k = 0$ we get 
$$2 s \frac{d U_0}{ds} + 2 s \frac{\Theta_s}{\Theta} = 0,$$ which is solved by
$$U_0 = \Theta^{-\half} .$$
The solution of the $\ell$th transport equation is then, 
$$U_{\ell} = - \frac{U_0}{4 \ell s^{m + \ell}} \int_0^s U_0^{-1} s^{\ell + m -1} \Box U_{\ell - 1} ds. $$
Hence we have a formal solution with the singularity of the Green's function in the elliptic case, and by
comparison with the Euclidean case we see that it solves $\Box E = \delta_0$.

We now consider the necessary modifications in the case of even dimensional spacetimes.
In  this case, $\Gamma^{\frac{2 - m}{2}} \Gamma^j$ is always an integer power. If we could solve
the transport equation for 
$j = \frac{m - 2}{2}$, the resulting term would be regular  with power $\Gamma^0$. The problem is
that 
$\Gamma^0$ should actually be a term with a logarithmic singularity $\log \Gamma$.

Thus the  parametrix \eqref{FIRST} is inadequate in even spacetime dimensions. Hadamard therefore introduced
 a logarithmic  term $V \log (\Gamma)$.  By a similar calculation to the above, 
$$ \begin{array}{l} \Box \left[ (\log \Gamma) V \right] 
=\left( -\Gamma^{-2} (4 \Gamma) + \Gamma^{-1}  (- 4  - \frac{J_r}{J}  2 r) \right) V
 + 2 \Gamma^{-1} (- 2 s \frac{d}{ds} V) + \log \Gamma \Box V. \end{array}$$
Due to \eqref{GAMMADOT},  all terms except the logarithmic term have the same singularity $\Gamma^{-1}$. 
On the other hand, the only way to eliminate the logarithmic term is to insist that
 $\Box V = 0$.  We further assume that
$$V = \sum_{j = 0}^{\infty}  V_j \Gamma^j.$$

We then return  to the unsolvable transport equations for $U_j$ for $j \geq \frac{m-2}{2}$, which
now acquires the new $V_0$ term to become:
$$\begin{array}{l} [ - 4\left(  (\frac{2-m}{2} + j)   (\frac{2-m}{2} + j-1)  \right) 
+(\frac{2-m}{2} + j)  (- 4  - \frac{J_r}{J}  2 r) )
 + 2 \left((\frac{2-m}{2}+ j) (- 2 s \frac{d}{ds}] \right)U_j 
+\Box U_{j-1}  \\ \\
+  \Gamma^{-1}  [ \left(4 +   (- 4  - \frac{J_r}{J}  2 r) \right) 
 + 2  \frac{d}{ds}]V_0 = 0 . \end{array}$$
When $j = \frac{m-2}{2}$, 
everything cancels in the $\Gamma^{-1}$ term except $\Box U_{m -1}$.  Hence, 
 we drop the $U_j$ for $j \geq \frac{m-2}{2}$  and assume the non logarithmic  part is
just the finite sum  $\sum_{j = 0}^{m - 1} U_j \Gamma^j$. 
But adding in the $V_0$ term we get  the transport equation,
$$ - 4s \frac{d V_0}{ds} - 2 r \frac{J_r}{J} V_0 = - \Box U_{m-1}. $$
Here, $U_{m-1}$ is known and we solve for $V_0$ to get,
$$V_{0} = - \frac{U_0}{4 s^{m }} \int_0^s U_0^{-1} s^{ m -1} \Box U_{m-1}ds. $$

The condition  $\Box V = 0$ imposed above then determines the rest of the coefficients $V_j$, 
$$V_{\ell} = - \frac{U_0}{4 \ell s^{m + \ell}} \int_0^s U_0^{-1} s^{\ell + m -1} \Box V_{\ell - 1} ds. $$
We now have two equations: the original $\Box (U \Gamma^{\frac{2-m}{2}} U + V \log \Gamma) = 0$
and the new $\Box V = 0$.   By solving the transport equations for $U_0, \dots, U_{m-1}, V_0, V_j (j \geq 1)$
we obtain a solution of an inhomogeneous equation of the form,
$$\Box (U \Gamma^{\frac{2 - m}{2}} + V \log \Gamma) = \sum_{j = 0} w_j \Gamma^j, $$
where the right side is regular. 
To complete the construction, we add a new  term of the form $W  = \sum_{\ell = 1}^{\infty} W_{\ell} (r^2 - t^2)^{\ell}$ 
in order to ensure that 
$$ \Box \sum_{j = 0}^{m-1} U_j (r^2 - t^2)^{-m + j} + V \log (r^2 - t^2) + W) = 0 $$
away from the characteristic conoid. It then
suffices to find $W_j$ so that
$$\Box \sum_{j = 1}^{\infty} W_j \Gamma^j =  \sum_{j = 0} w_j \Gamma^j. $$
This leads to more transport equations which are always solvable (by the Cauchy-Kowalevskaya theorem).
This concludes the sketch of the proof of Theorem \ref{HADAMARD}.

\subsection{\label{HADAN} Convergence in the real analytic case}

The above parametrix construction was formal.  However, when
the metric is real analytic,  Hadamard proved that the formal
 series converges for   $|t|$ and $|\Gamma|$ sufficiently small.   
The
convergence proof based on  the method of majorants.

\begin{theo} \label{MAINHAD} \cite{H} (see also \cite{Gar}) Assume that $(M, g)$ is real analytic.  Then there exists $K >0$
so that the Hadamard parametrix  converges for any
$(t, y)$ such that $t \not= 0$, $r(x, y) < \epsilon =  inj(x_0)$
and
\begin{equation} \label{ROFC} |t^2 - r^2| \leq \frac{\left((1 -
\frac{||y||}{\epsilon} \right)^2}{\left(1 + \frac{m_1}{\epsilon} +
\frac{m_1^2}{\epsilon^2} \right) K}, \;\;\;\;\;\;\;(m_1 =
\frac{m-2}{2}).
\end{equation}
 \end{theo}

It follows that the Hadamard fundamental solutions holomorphically extend to a neighborhood
of $\ccal_{\C}$ as branched meromorphic functions iwith $\ccal_{\C}$ as branch locus. To obtain single valued distributions, one
then needs to restrict the kernels to regions where a unique branch can be defined.

\subsection{Away from $\ccal_{\R}$}

A further complication is that the fundamental solution has only been constructed in  
a neighborhood of $\ccal_{\R}$.   But it is known to be real
analytic on $(\R \times M \times M) \backslash \ccal_{\R}$ and that it extends to a holomorphic kernel away from
a neighborhood of $\ccal_{\C}$. To prove this it suffices to note that  the analytic wave front 
set of the fundamental solution is $\ccal_{\R}$.   A more 
detailed proof  is given  by Mizohata \cite{Miz,Miz2}
for  `elementary solutions', i.e. solutions $E(t, x, y)$ of $\Box E
= 0$ such as $\cos t \sqrt{\Delta}$ or $\frac{\sin t
\sqrt{\Delta}}{\sqrt{\Delta}}$ whose Cauchy data is either zero or
a delta function. Here, $\Box$ operates in the $x$ variable with
$y$ as a parameter. To analyze the wave kernels away from the
characteristic conoid, Mizohata makes the decomposition
\begin{equation} E(t, x, y) = E_N(t, x, y) + w_N(t, x, y) + z_N(t, x,
y), \end{equation} where \begin{itemize}

\item $\Box E_N = f_N$ with $f_N \in C^{N-1}(\R \times M \times
M)$. Also, $E_N(t, x, y) |_{t = 0 } = \delta_y + a(x, y)$ with
$a(x, y) \in C^{\omega}(M \times M)$ (in fact it is independent of
$y$);

\item $\Box w_N(t, x, y) = - f_N(t, x, y),$ with $ w(0, x, y) =
0$;

\item $\Box z_N = 0, z_N(0, x, y) = - a(x, y)$.

\end{itemize}

The same kind of decomposition applies to the Hadamard fundamental solution. The
term constructed by the parametrix method is  $E_N$.  By solving the above equations,
it is shown in \cite{Miz} that the sum is analytic away from $\ccal_{\R}$. One can see that
the Hadamard method is only a branched Laurent type expansion near $\ccal_{\R}$ by
considering the examples in \S \ref{CONSTCURV}  of the kernels for spaces of constant
curvature.

\section{\label{HADPWK} Hadamard parametrix for the Poisson-wave kernel}

We are most interested in the Hadamard parametrix for the half-wave kernel, which does not seem to have been
 discussed in the literature. We are more generally interested in the Poisson-wave semi-group
$e^{i(t + i \tau)\sqrt{\Delta}}$ for $\tau > 0$. The Poisson-wave lernel
\begin{equation} \label{POISSWAVE} U(t + i \tau, x, y) = \sum_j
e^{(i (t + i \tau) \lambda_j} \phi_j(x) \phi_j(y) \end{equation}
is a real analytic kernel which possesses an analytic extension
to a Grauert tube. Thus, there exists a non-zero analytic radius
$\tau_{an}
> 0$ so that the Poisson kernel admits a holomorphic extension
$U(t + i \tau, \zeta, y)$ to $M_{\tau} \times M$  for $\tau \leq
\tau_{an}$.  Since
\begin{equation} U(i \tau) \phi_{\lambda} = e^{- \tau \lambda}
\phi_{\lambda}^{\C}, \end{equation}
 the eigenfunctions analytically extend to the same
maximal tube as does $U(i \tau)$.

We would like to construct a Hadamard type parametrix for \eqref{POISSWAVE}. 
We may derive it from the  Feynman-Hadamard fundamental solution 
in Proposition \ref{UF} 
using that 
\begin{equation}   \frac{d}{dt}  \frac{e^{i |t| \sqrt{\Delta}}}{\sqrt{\Delta}} = i \mbox{sgn}(t)\; e^{i |t| \sqrt{\Delta}} \end{equation}
and 
\begin{equation} e^{i t \sqrt{\Delta}} = \frac{1}{i} H(t)    \frac{d}{dt}  \frac{e^{i |t| \sqrt{\Delta}}}{\sqrt{\Delta}} 
- \frac{1}{i} H(-t)    \frac{d}{dt}  \frac{e^ {-i |t| \sqrt{\Delta}}}{\sqrt{\Delta}} . \end{equation}
Hence, 
\begin{equation} \label{ddt} \frac{d}{i dt} U_F(t) = e^{i t \sqrt{\Delta}},  \;\; (t > 0).  \end{equation}
The restriction to $t > 0$ is consistent with the kernel analysis in Remark \ref{KERNELS}, reflecting the
fact that $e^{it \sqrt{\Delta}}(x, y)$ has the singularity $((t + i 0)^2 - r^2)^{\frac{ - m}{2}}$ (in odd spacetime
dimensions) while $U_F(t)$ has the singularity $(t^2 - r^2 + i0)^{\frac{2 - m}{2}}$. We note (again) that
$((t + i 0)^2  - r^2)^{\alpha} = (t^2 - r^2 +  i 0)^{\alpha}$ for $t > 0$.

From Theorem \ref{HADAMARD} we conclude: 
\begin{cor}  \label{HALFHAD} Let $(M, g)$ be real analytic. Then with the $U_j, V_k, W_{\ell}$ defined as in Theorem \ref{HADAMARD},
we have:
\begin{itemize}

\item In odd spacetime dimensions, for $t > 0$ the Poisson-wave kernel $U(t + i \tau, x, y)$ ($\tau > 0$)
has the form   $ A  \Gamma^{\frac{ - m}{2}}$ where $A
 = \sum_{j = 0}^{\infty} A_j \Gamma^j$ with $A_j$ holomorphic. The series converges absolutely to a holomorphic function for $|\Gamma | < \epsilon$ sufficient
small, i.e. near the characteristic conoid. 

\item In even spacetime dimensions,  for $t > 0$,
the Poisson-wave kernel has the form $ B \Gamma^{\frac{ - m}{2}} + C \log \Gamma +D $ where the 
coefficients $B, C, D$ are holomorphic in a neighborhood of $\ccal_{\C}$, and have the same $\Gamma$
expansions as $A$.

\end{itemize}

\end{cor}

We use this parametrix to prove Theorem \ref{PTAULWL}  (1). 

\subsection{\label{OSCHAD} Hadamard parametrix as an oscillatory integral with complex phase}

 Corollary \ref{HALFHAD} gives a precise description of the singularities of the Poisson-wave
propagator. It  implicitly describes the kernel as a Fourier integral kernel. We
now make this description explicit in the real domain. In the following sections, we extend the
description to the complex domain.

We first express 
$\Gamma^{\frac{-m}{2} + j}$  as an
 oscillatory integral with one phase variable using the well-known identity
\begin{equation} \label{INT} \int_0^{\infty} e^{i \theta \sigma} \theta_+^{\lambda} d \lambda
= i e^{i \lambda \pi/2} \Gamma(\lambda + 1) (\sigma + i
0)^{-\lambda - 1}.
\end{equation}
At least formally, this leads to the representation $$  \int_{0}^{\infty} e^{i \theta (t^2 - r^2)}
 \theta_+^{\frac{n-1}{2} - j}
d\theta  = i e^{i (\frac{n-1}{2} - j) \pi/2} \Gamma(\frac{n-1}{2} -
j + 1) (t^2 - r^2 + i0)^{j-\frac{n -1}{2} - 1}
$$ for the principal term of the Poisson-wave.  Here, the notation $\Gamma = t^2 - r^2$
unfortunately clashes with 
that for the Gamma function, and we temporarily write out its defintion. 

In even space dimensions, the Hadamard parametrix for the Hadamard-Feynman fundamental solution thus 
has the form
\begin{equation}  \label{SUMUJ} \sum_{j = 0}^{\infty} U_j(t, x, y) \Gamma^{\frac{1 - n}{2} + j} =   \int_{0}^{\infty} e^{i \theta (t^2 - r^2)}
\left(\sum_{j = 0}^{\infty} U_j(t, x, y)  (i e^{i (\frac{n-1}{2} - j) \pi/2})^{-1}  \frac{ \theta_+^{\frac{n-3}{2} - j}}{\Gamma(\frac{n-3}{2} - j + 1)} \right)
d\theta. \end{equation} 
Here we follow Hadamard's notation, but it is simpler to re-define the coefficients $U_j$ so that the $\Gamma$-factors
appear on the left side as in \cite{Be} (7).  We thus define
$$\ucal_j(t, x, y) =  \left((i e^{i (\frac{n-1}{2} - j) \pi/2})^{-1}  \frac{1}{\Gamma(\frac{n-3}{2} - j + 1)} \right) U_j(t, x, y) . $$
By the  duplication formula
$\Gamma(z) \Gamma(1 - z) = \frac{\pi}{\sin \pi z}$ with  $z = \frac{m}{2} - k - \frac{\alpha}{2}$, i.e. 
$$\Gamma(\frac{m}{2} - j  - \frac{\alpha}{2}) =  (-1)^j \frac{\pi}{\sin \pi (\frac{m}{2}  - \frac{\alpha}{2})}   \frac{1}{\Gamma(- \frac{m}{2} + 1 + j  + \frac{\alpha}{2})},$$ it follows that
$$U_j(t, x, y) = \left( (-1)^j \frac{\pi}{\sin \pi (\frac{m}{2}  - \frac{\alpha}{2})}   \frac{1}{\Gamma(- \frac{m}{2} + 1 + j  + \frac{\alpha}{2})} \right)
\ucal_j(t, x, y), $$
so that the formula in odd spacetime dimensions becomes 
\begin{equation} \label{HDSTUFFa} \begin{array}{ll}  C_n \frac{1}{\sin \pi (\frac{m}{2}  - \frac{\alpha}{2})} 
\sum_{j = 0}^{\infty} (-1)^j \frac{\ucal_j(t, x, y) }{\Gamma(- \frac{m}{2} + 1 + j  + \frac{\alpha}{2})} (t^2 - r^2)^{-\frac{m - 2}{2} + j} \\ & \\
& = \int_0^{\infty} e^{i \theta(t^2 - r^2)}
\left(\sum_{j = 0}^{\infty} \ucal_j(t, x, y)  \theta_+^{\frac{n-3}{2} - j
}\right)
d\theta. \end{array} \end{equation} 

The  amplitude in the right side of \eqref{HDSTUFFa} is then a formal analytic symbol,
\begin{equation} \label{AMPa} A(t, x, y, \theta) =  \sum_{j =
0}^{\infty} \ucal_j (t, x,y) \theta_+^{\frac{n-3}{2} - j},
\end{equation} 
Due to the Gamma-factors appearing in the identity \eqref{INT},  convergence of the series on the left side of \eqref{HDSTUFFa} does not imply convergence of the series
\eqref{AMPa}.  However,  there exists a realization of the formal symbol \eqref{AMPa}
by a holmorphic symbol 
$$\acal(t, x, y, \theta) = \sum_{0 \leq j \leq \frac{\theta}{e C}}   \ucal_j (t, x,y) \theta_+^{\frac{n-3}{2} - j},$$
and one obtains an analytic 
parametrix 
\begin{equation} \label{PARAONEanh} U(t, x, y) = \int_0^{\infty} e^{i \theta \Gamma} \acal(t, x, y, \theta) d \theta
\end{equation}
which approximates the wave kernel  for small $|t|$ and $(x, y)$ near the diagonal up to a holomorphic
error, whose amplitude is exponentially decaying in $\theta$. Here, we recall (see \cite{Sj}, p. 3 and section 9)
that a  {\it classical formal analytic symbol} (\cite{Sj}, page 3)  on a domain $\Omega \subset \C^n$ is a formal 
semi-classical series
$$a(z, \lambda) = \sum_{k = 0}^{\infty} a_k(z) \lambda^{-k}, $$
where $a_k(z, \lambda)  \in \ocal(\Omega)$ for all $\lambda > 0$.
Then for some $C > 0$,  the $a_k(z) \in \ocal(\Omega)$ satisfy
$$|a_k(z)| \leq C^{k + 1} k^k, \;\; k = 0, 1, 2, \dots. $$
A realization of the formal symbol is a genuine holomorphic symbol of the form,
$$a(z, \lambda) = \sum_{0 \leq k \leq \frac{\lambda}{e C}} a_k(z)
\lambda^{-k} . $$ It is an analytic symbol since, with the
index restriction,
$$|a_k(z) \lambda^{-k}| \leq  C_{\Omega} (\frac{C k}{\lambda})^k \leq C
e^{-k}. $$ Hence the series converges uniformly on $\Omega$ to a
holomorphic function of $z$ for each $\lambda$.

Returning to \eqref{AMPa}, the   Hadamard-Riesz coefficients $\ucal_j$ are determined inductively
by the transport equations
\begin{equation}\left\{\begin{array}{l}
 \frac{\Theta'}{2 \Theta} \ucal_0 + \frac{\partial \ucal_0}{\partial r} = 0\\ \\
4 i r(x,y) \{(\frac{k+1}{r(x,y)} +  \frac{\Theta'}{2 \Theta})
\ucal_{j+1} + \frac{\partial \ucal_{j + 1}}{\partial r}\} = \Delta_y \ucal_j.
\end{array} \right., \end{equation}
whose solutions are given by:
\begin{equation}\label{HR} \left\{ \begin{array}{l} \ucal_0(x,y) = \Theta^{-\half}(x,y) \\ \\
\ucal_{j+1}(x,y) =  \Theta^{-\half}(x,y) \int_0^1 s^k \Theta(x,
x_s)^{\half} \Delta_2 \ucal_j(x, x_s) ds
\end{array} \right. \end{equation}
where $x_s$ is the geodesic from $x$ to $y$ parametrized
proportionately to arc-length and where $\Delta_2$ operates in the
second variable.

As discussed above,   the representation \eqref{HDSTUFFa}  does not suffice when   $n$ is odd, since  $
\Gamma(z)$ and $\theta_+^{z}$ have poles at the negative integers.  To  rescue the
representation when $n$ is odd, we need to use the distributions $\theta_+^{-n}$ with $n = 1, 2, \dots$,
defined as follows  (see \cite{Ho} Vol. I):   
$$\theta_+^{-k}(\phi) = \int_0^{\infty} (\log \theta) \phi^{(k)}(\theta)
dx/(k-1)! + \phi^{(k-1)}(0) (\sum_{j = 1}^k 1/j)/(k-1)!. $$
This family behaves in an unusal way under derivation, 
$$\frac{d}{d\theta} \theta_+^{-k} = -k \theta_+^{-k-1} + (-1)^k \delta_0^{(k)}/k! $$
(see \cite{Ho} Vol. I (3.2.2)'') and is therefore sometimes avoided in the Hadamard-Riesz parametrix construction (as in \cite{Be}).

However, we have already constructed the parametrices and only want to express them in terms of the
above oscillatory integrals to make contact with Fourier integral operator theory.  In odd space dimensions,
the Hadamard parametrices  can be written in the 
 form
\begin{equation} \label{ODDPAR} \begin{array}{l} \int_0^{\infty} e^{i \theta \Gamma} \left(
U_0(t, x, y)  \theta_+^{ m} + \cdots + U_{m} \theta_+^0 \right) d \theta \\ \\
+ \int_0^{\infty} e^{i \theta\Gamma}  \left( U_{m + 1} \theta_+^{-1} + U_{m + 2} \theta_+^{-2} + \cdots \right) d \theta 
\end{array} \end{equation}
Again the amplitude is a formal symbol. To produce a genuine amplitude it needs to be replace by a realization
which approximates it modulo a holomorphic symbol which is exponentially decaying in $\theta$. 

We are paying close attention to the regularization of the integral at $\theta = 0$, but
only the behavior of the amplitude as $\theta \to \infty$ is relevant to the singularity. 
The terms with $\theta_+^{- k}$ for $k > 0$ produce logarithmic terms in the kernel. 
If  we use a smooth cutoff  at $\theta = 0$, we obtain distributions of the form 
$$u_{\mu}(\Gamma) = \int_{\R} e^{i  \theta \Gamma} \chi(\theta) \theta^{\mu} d\theta$$
where $\chi(\theta) = 1$ for $\theta \geq 1 $ and $\chi(\theta) = 0$ for $\theta\leq \half$. Then
$$u_{-k} (\Gamma) = i^{k+1} \Gamma^{k-1} \log \Gamma, \;\;\mbox{ modulo}\;\;C^{\infty}, \;\;\;
u_{-k}( - \Gamma) = (-i)^{n+1} \Gamma^{n-1} \log \Gamma. $$
Hence the terms with negative powers of $\theta_+$ in \eqref{ODDPAR} produce the logarithmic terms and 
the holomorphic terms.

Above, we discussed the Hadamard-Feynman fundamental solution, but for $t > 0$ we only need to differentiate
it in $t$ (according to Proposition \ref{HALFHAD} ) to obtain the parametrices for the Poisson-wave group.  Away from
the characteristic conoid the Schwartz kernels of the Poisson-wave group and Hadamard-Feynman fundamental
solution are holomorphic by the theorem on propagation of analytic wave front sets \cite{Ho,Sj} (see also \cite{Miz, Miz2}). 
The Fourier integral structure and mapping properties follow immediately from the order of the amplitude and from
the exact formula for the phase.

\subsection{\label{MODOSCHAD} Modified Hadamard parametrix with phase $t - r$ when $t \geq 0$}

We now assume that $t \geq 0$ as well as $\tau > 0$. We note that the phase $t^2 - r^2$ for $U(t, x,y)$
factors as $(t - r) (t + r)$ with $t + r \geq 0$ when $t \geq 0$. Of course, $r(x, y) \geq 0$ in the real domain. 
We can therefore change variables $\theta \to (t + r) \theta$ in \eqref{PARAONEanh} to obtain the modified
Hadamard parametrix,

\begin{equation} \label{PARAONEanh2} U(t, x, y) = (t + r)^{-1} \int_0^{\infty} e^{i \theta (t - r)} \acal(t, x, y, 
\frac{\theta}{t + r}) d \theta, \;\;\;\; (t \geq 0)
\end{equation}
with phase $t - r$. We note that it is singular when $r = 0$ but we only intend to use it for $r_{\C} \not= 0$. 
The amplitude $= (t + r)^{-1}   \acal(t, x, y, 
\frac{\theta}{t + r}) $ is a representative of an analytic symbol as long as $r + t \not= 0$ and $r \not= 0$.

\section{H\"ormander parametrix for the Poisson-wave kernel \label{HORPAR}}

A more familiar  construction of $U(t, x, y)$ and its analytic continuation  which is particularly useful for small
$|t|$ is  the one \eqref{FOURIERINV} based on the Fourier inversion formula. Its generalization to 
Riemannian manifolds is given by 
\begin{equation} \label{PARAONEa} U(t, x, y) = \int_{T^*_y M} e^{
i t |\xi|_{g_y} } e^{i \langle \xi, \exp_y^{-1} (x) \rangle} A(t,
x, y, \xi) d\xi,
\end{equation}
for $(x, y)$ sufficiently close to the diagonal.  We use this parametrix to prove Theorem \ref{PTAULWL} (2).

The amplitude is a polyhomogeous symbol of the form
\begin{equation} \label{AMP} A(t, x, y, \xi)  \sim \sum_{j = }^{\infty} A_j(t, x, y, \xi), \end{equation}
where the asymptotics are in the sense of the symbol topology and where
$$A_j(t, x, y, \tau \xi) = \tau^{-j} A_j(t, x, y, \xi), \;\;\; \mbox{for} |\xi| \geq 1. $$
The principal term $A_0(t, x, y, \xi)$ equals $1$ when $t = 0$ on the diagonal, and the higher $A_j$ are determined
by transport 
equations  discussed  in \cite{DG}.  

It can be verified that in the case of real analytic $(M, g)$,  the amplitude is a classical formal analytic  symbol
(see \S \ref{OSCHAD}). 
Hence if $\acal(t, x, y, \xi)$ is a realization of the amplitude $A(t, x, y, \xi)$, then one obtains an analytic 
parametrix 
\begin{equation} \label{PARAONEan} U(t, x, y) = \int_{T^*_y M} e^{
i t |\xi|_{g_y} } e^{i \langle \xi, \exp_y^{-1} (x) \rangle} \acal(t,
x, y, \xi) d\xi,
\end{equation}
which approximates the wave kernel  for small $|t|$ and $(x, y)$ near the diagonal up to a holomorphic
error, whose amplitude is exponentially decaying in $|\xi|$.

\subsection{Subordination to the heat kernel}

The parametrix \eqref{PARAONE} can also be obtained by subordinating the Poisson-wave kernel to
the heat kernel in the sense of \S \ref{SUB}. To make use of it, one needs to analytically the heat kernel
to $M_{\tau}$. This analytic continuation was studied by 
Golse-Leichtnam-Stenzel in \cite{GLS}., who proved the following: For any $x_0 \in M$ there exists $\epsilon, \rho > 0$
and an open neighborhood $W$ of $x_0$ in $M_{\epsilon}$ such that for $0 < t < 1$ and $(x, y)
\in W \times W$, 
$$E(t, x, y) = N(t, x, y) e^{- \frac{r^2(x, y)}{4t}} + R(t, x, y), $$
where $$N(t, x, y) = \sum_{0 \leq j \leq \frac{1}{C t}} W_j(x, y) t^j  $$
as $t \downarrow 0^+$ where $W_j(x, y)$ are the Hadamard-Minakshisundaram-Pleijel
heat kernel coefficients.  is an analytic symbol of ordern $n/2$ with respect to $t^{-1}$ in the sense of
\cite{Sj}.  As above, the remainder is exponentially small, 
$$|R(t, x, y) | \leq C e^{- \frac{\rho}{8 t}} $$
with a uniform $C$ in $(x, y)$ as $t \downarrow 0^+$. 
The heat kernel itself obviously admits a holmoorphic extension in the open subset  $\Re r_{\C}^2(x, y) > 0$
ot $M_{\C} \times M_{\C}$.

\section{Complexified Poisson kernel as a complex Fourier
integral operator}

We now consider the Fourier integral operator aspects of the analytic continuation of the
Poisson-wave kernel $U(t + i \tau, \zeta, y)$ for $\tau > 0$ and $(\zeta, y) \in M_{\tau}  \times M$,
developing analogues of the results of \S \ref{OSCHAD}  in the complex domain.
  The aim is to prove that the analytically continued
Poisson-wave kernel is a complex Fourier integral operator.
  We
denote by $\ocal (M_{\epsilon})$ the space of holomorphic
functions on the Grauert tube and by a slight abuse of notation we also
denote by $\ocal(\partial M_{\epsilon})$ the CR holomorphic functions on the boundary $\partial M_{\epsilon}$  of the strictly
pseudo-convex domain $M_{\epsilon}$ (the null space of the boundary Cauchy-Riemann operator $\dbar_b$.)
 In particular, we denote by  $\ocal^0(\partial M_{\epsilon}) = H^2(\partial M_{\epsilon})$
the Hardy space of boundary values of holomorphic functions of
$M_{\epsilon}$ which lie in $L^2(\partial M_{\epsilon})$ relative
to the natural Liouville measure
\begin{equation} \label{LIOUVILLE} d\mu_{\tau} = (i \ddbar
\sqrt{\rho})^{m-1} \wedge d^c \sqrt{\rho}. \end{equation}
 We further  denote by  $\ocal^{s +
\frac{n-1}{4}}(\partial M _{\tau})$ the Sobolev spaces of CR
holomorphic functions on $\partial M_{\tau}$, i.e.
\begin{equation} \label{SOBSP} {\mathcal O}^{s +
\frac{m-1}{4}}(\partial M_{\tau}) = W^{s + \frac{m-1}{4}}(\partial
M_{\tau}) \cap \ocal (\partial M_{\tau}), \end{equation}  where
$W_s$ is the $s$th Sobolev space.

The spray \begin{equation} \label{SIGMATAU} \Sigma_{\tau} =
\{(\zeta, r d^c \sqrt{\rho}(\zeta): r \in \R_+\} \subset T^*
(\partial M_{\tau})
\end{equation}  of the contact form $d^c \sqrt{\rho}$ defines a
symplectic cone. There exists a symplectic equivalence  (cf.
\cite{GS2})
\begin{equation} \iota_{\tau} : T^*M - 0 \to
\Sigma_{\tau},\;\; \iota_{\tau} (x, \xi) = (E(x, \tau
\frac{\xi}{|\xi|}), |\xi|d^c \sqrt{\rho}_{E(x, \tau
\frac{\xi}{|\xi|})} ).
\end{equation}

 The following theorem is stated in \cite{Bou} (see also
 \cite{Z3}):

\begin{theo}\label{BOUFIO2} (see  \cite{Bou, GS2,GLS})   For sufficiently small $\tau > 0$, $  U_{\C} (i \tau): L^2(M)
\to \ocal(\partial M_{\tau})$ is a  Fourier integral
operator of order $- \frac{m-1}{4}$  with complex phase  associated to the canonical
relation
$$\Lambda = \{(y, \eta, \iota_{\tau} (y, \eta) \} \subset T^*M \times \Sigma_{\tau}.$$
Moreover, for any $s$,
$$ U_{\C} (i \tau): W^s(M) \to {\mathcal O}^{s +
\frac{m-1}{4}}(\partial  M_{\tau})$$ is a continuous isomorphism.
\end{theo}

The proof of Theorem \ref{BOUFIO2} is barely sketched in
\cite{Bou}. However, the theorem  follows almost immediately  from  the construction of the branched
meromorphic  Hadamard parametrix
in Corollary \ref{HALFHAD}, or alternatively from the  analytic continuation of the H\"ormander parametrix of \S \ref{HORAC}.  
It suffices to show that either is a parametrix for $U_{\C}(i \tau, \zeta, y)$, i.e. differs from it by an analytic kernel (smooth
would be sufficient by analytic wave front set considerations). But the Hadamard parametrix construction is an
exact formula and actually gives a more precise description of the singularities of $U_{\C}(i \tau, \zeta, y)$ than
is stated in Theorem \ref{BOUFIO2}.  We briefly explain how either the Hadamard or H\"ormander parametrix
can be used to complete the proof.

\subsection{Fourier integral distributions  with complex phase}

First,  we  review  the relevant definitions (see \cite{Ho} IV, \S 25.5 or \cite{MeSj}).  A Fourier integral distribution
with complex phase on a manifold $X$ is a distribution that can locally be represented by an oscillatory integral
$$A(x) = \int_{\R^N}  e^{i \phi(x, \theta)} a(x, \theta) d \theta$$
where $a(x, \theta) \in S^m( X \times V)$ is a symbol of order $m$ in a cone $V \subset \R^N$ and where the phase $\phi$ is a positive
regular phase function, i.e. it satisfies
\begin{itemize}

\item $\Im \phi \geq 0$;

\item $d\frac{\partial \phi}{\partial \theta_1}, \dots, d \frac{\partial \phi}{\partial \theta_N}$ are
linearly independent complex vectors on $$C_{\phi \R} = \{(x, \theta) : d_{\theta} (x, \theta) = 0\}. $$

\item In the analytic setting (which is assumed in this article), $\phi$ admits an analytic continuation 
$\phi_{\C}$ to an open
cone in 
$X_{\C} \times V_{\C}$.

\end{itemize}

Define  
$$C_{\phi_{\C}} = \{(x, \theta) \in X_{\C} \times V_{\C} : \nabla_{\theta} \phi_{\C}(x, \theta) = 0\}. $$
Then $C_{\phi_{\C}}$ is a manifold near the real domain. One defines the Lagrangian submanifold
$\Lambda_{\phi_{\C}} \subset T^* X_{\C}$ as the image 
$$(x, \theta) \in C_{\phi_{\C}} \to (x, \nabla_x \phi_{\C}(x, \theta)). $$

\subsection{\label{AC1}Analytic continuation of the Hadamard parametrix}

As in \S \ref{OSCHAD} and \S \ref{MODOSCHAD}, we can express $U_{\C}(i \tau, \zeta, y)$ as a local Fourier integral
distribution with complex phase  by rewriting  the Hadamard
series in Corollary \ref{HALFHAD} as oscillatory integrals.  Here we assume that $\tau > 0, t \geq 0$. 

A complication  is that we can only use  the complexified phase  $\Gamma = t^2 - r^2$ in regions
of complexified $\R \times M \times M$ where its imaginary part is $\geq 0$.  As in \S \ref{MODOSCHAD},
we  could also use the
phase $t - r$  (resp.  $t + r$) in regions where $t + r \not= 0$ (resp. $t - r \not= 0$) and where the contour
$\R_+$ can be deformed back to itself after the 
the change of variables $\theta \to (t + r) \theta$.

\subsection{\label{HORAC} Analytic continuation of the H\"ormander parametrix}

As was the case in $\R^n$ \eqref{PWAC},
the parametrix (\ref{PARAONEan})  admits an analytic
continuation in time to a strip $\{t + i \tau: \tau < \tau_{an}, |t| < 1\}$.  In the space
variables, the   parametrix  then  admits
an analytic continuation to complex $x, y$ satisfying $|r_{\C}(x,
y) | \leq \tau. $

  The analytically continued parametrix \eqref{PARAONE} approximates 
the true analytically continued Poisson kernel  up to a holomorphic kernel. 
 More preicsely,  for any $x_0 \in M$  and $\tau > 0$, there exists $\epsilon, \rho > 0$
and an open neighborhood $W$ of $x_0$ in $M_{\tau}$ such that for $|t| < 1$ and $(x, y)
\in W \times W$, 
\begin{equation}  \label{PARAONE} U(t + i \tau , x, y) =  \int_{T^*_y M} e^{
- \tau |\xi|_{g_y} } e^{i \langle \xi, \exp_y^{-1} (x) \rangle}
\acal (t + i \tau, x, y, \xi) d\xi + R(t, x, y), \end{equation}
where $R(t, x, y)$ is holomorphic for small $|t|$ and for $(x, y)$ near the diagonal. 

The parametrix is only defined near the diagonal where $\exp_y^{-1}$ is defined. However one can extend
it to a global  holomorphic kernel away from $\ccal_{\C}$ by cutting off the first term of \eqref{PARAONE} with a smooth cutoff
$\chi(x, y)$ supported near the diagonal in $M_{\tau} \times M_{\tau}$ and then  solving a $\dbar$ problem
on the Grauert tube (or a $\dbar_b$ problem on its boundary) to  extend the kernel to be globally holmorphic (resp. CR).
We refer to \cite{Z1} for a more detailed discussion. This gives an alternative to the Hadamard parametrix construction
of Corollary \ref{HALFHAD}.

 This concludes
the sketch of proof of Theorem \ref{BOUFIO2}.

\section{\label{POISSON} Tempered spectral projector and Poisson
semi-group as complex Fourier integral operators}

To study the  tempered  spectral projection kernels
(\ref{CXSPMa}), we further need to  continue $U_{\C}(t, \zeta, y)$
anti-holomorphically in the $y$ variable. The discussion is
similar to the holomorphic case except that we need to double the
Grauert tube radius to obtain convergence. We thus have,
\begin{equation} \label{CXWVGP} \begin{array}{lll} U_{\C} (t + 2 i \tau, \zeta, \bar{\zeta}) &= & \sum_j
e^{(- 2 \tau + i t) \lambda_j} |\phi_j^{\C}(\zeta)|^2 \\ && \\ & =
& \int_{\R} e^{i t \lambda} d_{\lambda} P_{[0, \lambda]}^{\tau}
(\zeta, \bar{\zeta}). \end{array} \end{equation} Properties of
these kernels may be obtained from kernels which are analytically
continued in one variable only from  the formula,
\begin{equation} \label{EFORM}
\begin{array}{lll} U_{\C} (t + 2 i \tau, \zeta, \bar{\zeta}') &  = &
\int_M  U (t + i \tau, \zeta, y)  U_{\C} (i \tau, y, \bar{\zeta}'
) dV_g(x)\\
&& \\
&&  = \sum_j  e^{(- 2 \tau + i t) \lambda_j} \phi_j^{\C}(\zeta)
\overline{\phi_j^{\C}(\zeta')}.
\end{array}
\end{equation}

We have,

\begin{prop} \label{FINALHAD} For small $t,  \tau > 0$ and for
sufficiently small $\tau \geq \sqrt{\rho}(\zeta) > 0$, there exists a realization $\bcal(t, \zeta, \bar{\zeta}, \theta)$ of
a formal analytic symbol $B(t, \zeta, \bar{\zeta}, \theta)$ so that as
tempered distributions on $\R \times M_{\tau}$,
\begin{equation} \label{HDD} U_{\C} (t + 2 i \tau, \zeta, \bar{\zeta}) = \int_0^{\infty} e^{i \theta ((t
+ 2 i \tau) - 2 i \sqrt{\rho}(\zeta))} \bcal(t, \zeta,
\bar{\zeta}, \theta) d \theta + R(t + 2 i \tau, \zeta, \bar{\zeta}),
\end{equation}
where 
$R(t + 2 i \tau, \zeta, \bar{\zeta}) $ is the restriction to the anti-diagonal of a  holomorphic kernel. Moreover
\begin{itemize}

\item $\theta ((t
+ 2 i \tau) - 2 i \sqrt{\rho}(\zeta))$ is a phase of positive type.  

\item If $\sqrt{\rho}(\zeta) < \tau$ the entire kernel
is locally holomorphic.  

\item If $\sqrt{\rho}(\zeta) = \tau$  then
\begin{equation} \label{HDDb} U_{\C} (t + 2 i \tau, \zeta, \bar{\zeta}) = \int_0^{\infty} e^{i \theta t} \bcal(t, \zeta,
\bar{\zeta}, \theta) d \theta + R(t + 2 i \tau, \zeta, \bar{\zeta}).
\end{equation}

\end{itemize}

\end{prop}

\begin{proof}  We use the Hadamard parametrix (Corollary \ref{HALFHAD}) for $U(t + 2 i \tau, \zeta, \bar{\zeta})$ and
use \eqref{rhoeq} to simplify the phase, i.e. we write 
$$\Gamma(t + 2  i \tau, \zeta, \bar{\zeta}) = (t + 2 i \tau -2 i \sqrt{\rho})(t + 2 i \tau + 2 i \sqrt{\rho}) $$
in the Hadamard parametrix in Corollary \ref{HALFHAD}.  The factors of   $(t + 2 i \tau + 2 i \sqrt{\rho}) $ 
are non-zero when $\tau > 0$ and can be absorbed into the Hadamard coefficients. We denote the
new amplitude by $\bcal$ to distinguish it from the amplitude in Corollary \ref{HALFHAD}.  We then express
each term  as a Fourier integral  distribution of complex type with phase $t + 2 i \tau - 2 i
\sqrt{\rho}$. It is manifestly of positive type. On $\partial M_{\tau}$, $t + 2 i \tau - 2i \sqrt{\rho}$ simplifies
to $t$.

\end{proof}

\subsection{Complexified wave group and \szego kernels}

As in \cite{Z3} it will also be necessary for us to understand the composition   $U_{\C}(i \tau)^* U_{\C}(i \tau) $. 
In this regard, it is useful to introduce the
  \szego kernels $\Pi_{\tau}$ of $M_{\tau}$, i.e.  the orthogonal
projections

\begin{equation} \Pi_{\tau}: L^2(\partial M_{\tau}, d\mu_{\tau}) \to H^2(\partial M_{\tau},
d\mu_{\tau}), \end{equation} where $d\mu_{\tau}$ is the natural
 volume form \eqref{LIOUVILLE}. Here as above,
$H^2(\partial M_{\tau}, d\mu_{\tau})$ is the Hardy space of
boundary values of holomorphic functions in $M_{\tau}$ which
belong to $ L^2(\partial M_{\tau}, d\mu_{\tau})$. It is  simple to
prove that the restrictions of  $\{\phi_{\lambda_j}^{\C}\}$ to
$\partial M_{\tau}$  is a basis of $H^2(\partial M_{\tau},
d\mu_{\tau})$. The \szego projector $\Pi_{\tau}$ is a complex
Fourier integral operator with a positive complex canonical
relation. The
 real points of its canonical relation form the graph
$\Delta_{\Sigma}$ of the identity map on the symplectic cone
 $\Sigma_{\tau}
\subset T^*
\partial M_{\tau}$ (\ref{SIGMATAU}). We refer to \cite{Z3} for
further background. We only  need the first statement in the
following:

\begin{lem}\label{PSIDOstuff}  Let  $\Psi^s(X)$ denote  the class of pseudo-differential operators of
order $s$ on $X$. Then,

\begin{itemize}

\item  $U_{\C}(i \tau)^* U_{\C}(i \tau) \in \Psi^{-
\frac{m-1}{2}}(M)$ with principal symbol $|\xi|_g^{- (
\frac{m-1}{2})}.$

\item  $U_{\C}(i \tau) \circ U_{\C}(i \tau)^* = \Pi_{\tau}
A_{\tau} \Pi_{\tau}$ where $A_{\tau} \in \Psi^{ \frac{m-1}{2}}
(\partial M_{\tau})$ has principal symbol $|\sigma|_g^{ (
\frac{m-1}{2})}$ as a function on $\Sigma_{\tau}$.

\end{itemize}

\end{lem}\begin{proof}

This follows from Proposition \ref{BOUFIO2}. The first statement
is a special case of the following Lemma from \cite{Z3} (Lemma
3.1):
 Let  $a \in S^0(T^*M-0)$.  Then for all $0 < \tau < \tau_{\max}(g)$,
we have: $$ U (i \tau)^* \Pi_{\tau} a \Pi_{\tau} U(i \tau) \in
\Psi^{- \frac{m-1}{2}}(M),
$$
with principal symbol equal to $a(x, \xi) \; |\xi|_g^{- (
\frac{m-1}{2})}.$

The second statement follows from Theorem \ref{BOUFIO2} and the
composition theorem for complex Fourier integral operators. We do
not use it in this article and refer to \cite{Z1} for the proof.
We note that
\begin{equation} U_{\C}(i \tau) \circ U_{\C}(i \tau)^* (\zeta, \zeta')  =
\sum_j e^{- 2 \tau \lambda_j} \phi_{\lambda_j}^{\C}(\zeta)
\overline{ \phi_{\lambda_j}^{\C}(\zeta')}.\end{equation}

\end{proof}

\section{\label{LWL} One term local Weyl law}

In this section, we prove Theorem \ref{PTAULWL} (1). To prove
the local Weyl law we employ  parametrices for the Poisson-wave kernel  adapted
to $e^{i( t + i \tau) \sqrt{\Delta}}$ for $\tau > 0$ which are best adapted to the complex
geometry.

\subsection{Proof of the local Weyl law}

\begin{proof}

As in the real domain,  we  obtain asymptotics of $P_{[0,
\lambda]}^{\tau}(\zeta, \bar{\zeta})$ by the Fourier-Tauberian
method of relating their asymptotics to  the singularities in the
real time $t$ of the Fourier transform (\ref{CXWVGP}). We refer to
\cite{SV} (see also the Appendix of \cite{Z1}) for background on
Tauberian theorems. We follow the classical argument of \cite{DG},
Proposition 2.1, to obtain the local Weyl law with remainder one
degree below the main term.

The proof is based on the oscillatory integral representation of Proposition \ref{FINALHAD}. 
We are working in the case where  $\sqrt{\rho}(\zeta) = \tau$ and hence can simplify it to
\eqref{HDDb}.

 We then
introduce  a cutoff function $\psi \in \scal(\R)$ with $\hat{\psi}
\in C_0^{\infty}$ supported in sufficiently small neighborhood of
$0$ so that no other singularities of $U_{\C}(t + 2 i \tau, \zeta,
\bar{\zeta})$ lie in its support. We also assume $\hat{\psi}
\equiv 1$ in a smaller neighborhood of $0$.  We then change
variables $\theta \to \lambda \theta$ and  apply the complex
stationary phase to the integral,
\begin{equation} \label{CXPARAONEb}\begin{array}{ll}
 \int_{\R} \hat{\psi}(t) e^{- i \lambda t}  U_{\C} (t + 2 i \tau,
\zeta, \bar{\zeta}) dt &  = \int_{\R}  \int_0^{\infty} \hat{\psi}(t) e^{- i \lambda t}   e^{i \theta t} \left(\bcal(t, \zeta,
\bar{\zeta}, \theta) d \theta + R(t + 2 i \tau, \zeta, \bar{\zeta})) \right) dt.
\end{array}
\end{equation}
The second $R$ term can be dropped since it is of order $\lambda^{- M}$ for all $M > 0$. In the first we change
variables $\theta \to \lambda \theta$ to obtain a semi-classical Fourier integral distribution of real type
with phase $e^{i \lambda t (\theta - 1)}$. The critical set consists of $\theta = 1, t = 0$. The
phase is clearly  non-degenerate with Hessian determinant one and inverse
Hessian operator $D^2_{\theta, t}$. Taking into account the factor of $\lambda^{-1}$ from the change of variables,
the stationary phase expansion gievs
\begin{equation}\label{EXPANSIONCaa}  \sum_j \psi(\lambda - \lambda_j) e^{- 2 \tau
\lambda_j} |\phi_j^{\C}(\zeta)|^2 \sim \sum_{k = 0}^{\infty}\lambda^{\frac{n-1}{2} - k} \omega_k(\tau; \zeta)
\end{equation}
where the coefficients $\omega_k(\tau, ]\zeta)$ are smooth for $\zeta \in \partial M_{\tau}$. However the
coefficients are not uniform as $\tau \to 0^+$ due to the factors of $(t + 2 i \tau + 2 i \sqrt{\rho}(\zeta))$
which were left in the denominators of the modified Hadamard parametrix. Since $t= 0$ at  the stationary 
phase point,  the resulting expansion is equivalent to one with the large parameter $\tau \lambda$ (or
$\sqrt{\rho}(\zeta) \lambda$). The uniform expansion is then 
\begin{equation}\label{EXPANSIONCa}  \sum_j \psi(\lambda - \lambda_j) e^{- 2 \tau
\lambda_j} |\phi_j^{\C}(\zeta)|^2 \sim \sum_{k = 0}^{\infty}
\left(\frac{\lambda}{\tau}
  \right)^{\frac{n-1}{2} - k} \omega_k(\zeta, \bar{\zeta}),
\end{equation}
where $\omega_j $ are smooth in $\zeta$, and $\omega_0 = 1$. The
remainder has the same form.

To complete the proof, we apply  the  Fourier Tauberian theorem
(see   the Appendix (\cite{SV}): Let $N \in F_+$ and let $\psi \in
\scal (\R)$ satisfy the conditions:  $\psi$ is  even,
$\psi(\lambda)
> 0$ for all $\lambda \in \R$,   $\hat{\psi} \in
C_0^{\infty}$, and $\hat{\psi}(0) = 1$. Then,
$$\psi * dN(\lambda) \leq A \lambda^{\nu} \implies |N(\lambda) - N *
\psi(\lambda)| \leq C A \lambda^{\nu}, $$ where $C$ is independent
of $A, \lambda$. We apply it twice, first in the region
$\sqrt{\rho}(\zeta) \geq C \lambda^{-1}$ and second in the
complementary region.

In the first region, we   let $N_{\tau, \zeta}(\lambda) = P_{\tau,
\lambda}(\zeta, \bar{\zeta})$. It is clear that for $\sqrt{\rho}
= \tau$, $N_{\tau, \zeta}(\lambda)$ is
  a monotone non-decreasing
    function of $\lambda$ of polynomial growth which   vanishes for $\lambda
  \leq
  0$.  For $\psi \in \scal$ positive,  even and with $\hat{\psi} \in
  C_0^{\infty}(\R)$ and $\hat{\psi}(0) = 1$, we have by
  (\ref{EXPANSIONCa}) that
  \begin{equation} \label{PSIEST} \psi * d N_{\tau, \zeta}(\lambda) \leq C \left(\frac{\lambda}{\tau}
  \right)^{\frac{n-1}{2}}, \end{equation}
  where $C$ is independent of $\zeta, \lambda$.
 It follows by the Fourier Tauberian theorem  that
  $$N_{\tau, \zeta}(\lambda) = N_{\tau, \zeta}(\lambda) * \psi
  (\lambda) + O\left(\frac{\lambda}{\tau}
  \right)^{\frac{n-1}{2}}. $$
  Further, by integrating (\ref{EXPANSIONCa}) from $0$ to $\lambda$
  we have
  $$N_{\tau, \zeta}(\lambda) * \psi
  (\lambda) =    \left(\frac{\lambda}{\tau}
  \right)^{\frac{n-1}{2}} \left(\frac{\lambda}{\frac{n-1}{2} + 1} + O(1) \right), $$
  proving (1).

To obtain uniform asymptotics in $\tau$  down to $\tau = 0$,  we use instead the analytic continuation of the
H\"ormander 
parametrix (\ref{PARAONE}). We choose local coordinates near $x$
and write $\exp_x^{-1}(y) = \Psi(x, y)$ in these local coordinates
for $y$ near $x$, and write the integral $T^*_yM$ as an integral
over $\R^m$ in these coordinates. The holomorphic extension of the
parametrix  to the Grauert tube $|\zeta| < \tau$ at time $t + 2 i
\tau$ has the form
\begin{equation} \label{CXPARAONE} U_{\C}(t + 2 i \tau,
\zeta, \bar{\zeta}) = \int_{\R^n} e^{(i t - 2\tau )  |\xi|_{g_y} }
e^{i \langle \xi, \Psi (\zeta, \bar{\zeta}) \rangle} A(t, \zeta,
\bar{\zeta}, \xi) d\xi.
\end{equation}

Again,  we use a cutoff function $\psi \in \scal(\R)$ with
$\hat{\psi} \in C_0^{\infty}$ supported in sufficiently small
neighborhood of $0$ so that no other singularities of $E(t + 2 i
\tau, \zeta, \bar{\zeta})$ lie in its support and so that
$\hat{\psi} \equiv 1$ in a smaller neighborhood of $0$. We write
the integral in polar coordinates and obtain
\begin{equation} \label{CXPARAONEc}\begin{array}{l}  \int_{\R} \hat{\psi}(t) e^{-i \lambda t} U_{\C} (t + 2 i \tau,
\zeta, \bar{\zeta})dt \\ \\
= \lambda^m \int_{0}^{\infty} \int_{\R} \hat{\psi}(t) e^{-i
\lambda t} \int_{S^{n-1}} e^{(i t - 2\tau ) \lambda r} e^{i r
\lambda \langle \omega, \Psi (\zeta, \bar{\zeta}) \rangle} A(t,
\zeta, \bar{\zeta}, \lambda r \omega ) r^{n-1} dr
d\omega.\end{array}
\end{equation}

We then apply complex  stationary phase to the $dr dt$ integral,
regarding
$$\int_{S^{n-1}} e^{i r \lambda \langle \omega, \Psi (\zeta,
\bar{\zeta}) \rangle} A(t, \zeta, \bar{\zeta}, \lambda r \omega )
r^{m-1}  d\omega$$ as the amplitude. When $\sqrt{\rho}(\zeta) \leq
\frac{C}{\lambda}$ the exponent is bounded in $\lambda$ and the
integral defines a symbol. Applying stationary phase again to the
$dt d\theta$ integral now gives
\begin{equation}\label{EXPANSIONC}  \sum_j \psi(\lambda - \lambda_j) e^{- 2 \tau
\lambda_j} |\phi_j^{\C}(\zeta)|^2 \sim \sum_{k = 0}^{\infty}
\lambda^{n - 1 - k} \omega_k(\zeta, \bar{\zeta}),
\end{equation}
where $\omega_k(\zeta, \bar{\zeta})$ is smooth down to
the zero section.

We apply the Fourier Tauberian theorem again, but this time with
the estimates
$$\psi * d N_{\tau, \zeta} (\lambda) \leq  C \lambda^{n-1}, $$
where $C$ is independent of $\zeta$. We conclude that
$$N_{\tau, \zeta} (\lambda) = C \lambda^n + O(\lambda^{n-1}), $$
proving (2).

\end{proof}

\begin{cor}\label{EASY}  For all $\zeta \in M_{\C}$, and with
$\tau = \sqrt{\rho}(\zeta)$,
$$ c \lambda^{\frac{n+1}{2}}
 \leq  P^{\tau}_{[0, \lambda]}(\zeta, \bar{\zeta}) \leq C
\lambda^{ n}. $$
\end{cor}

\subsection{Proof of Corollary  \ref{PWa}}

\begin{proof}  For the upper bound, we use that  
 $$\sup_{\zeta \in
\partial M_{\tau}}
|\phi_{\lambda}^{\C}(\zeta)|^2 \leq  \sup_{\zeta \in
\partial M_{\tau}} \Pi_{I_{\lambda}}(\zeta, \overline{\zeta})| \leq  \sup_{\zeta \in
\partial M_{\tau}} e^{ \lambda
\sqrt{\rho}(\zeta)}| P_{I_{\lambda}}(\zeta)|. $$
The upper bound stated in Corollary   \ref{PWa} then follows
from Corollary \ref{EASY} to  Theorem \ref{PTAULWL}.

For the lower bound in (2) of Corollary \ref{PWa}, we use that
$$||\phi^{\C}_j||_{L^2(\partial M_{\tau})} = e^{2 \tau_j} \langle
U(i \tau)^* U(i \tau) \phi_j, \phi_j \rangle_{L^2(M)}. $$ By Lemma
\ref{PSIDOstuff}, the operator $U(i \tau)^* U(i \tau)$ is an
elliptic pseudodifferential operator of order $\mu =
-\frac{n-1}{2}$ (or so). Let $C > 0$ be a lower bound for its
symbol times $\langle \xi \rangle^{ \mu}$. Then by Garding's
inequality, $$\langle U(i \tau)^* U(i \tau) \phi_j, \phi_j
\rangle_{L^2(M)} \geq C \lambda_j^{-\mu}, $$ and so
\begin{equation} \label{GARDING} ||\phi^{\C}_j||_{L^2(\partial M_{\tau})} \geq C
\lambda_j^{-\mu} e^{2 \tau \lambda_j}. \end{equation}

\end{proof}

\section{Siciak extremal functions: Proof of Theorem \ref{SICIAK}
(1)}

In this section we prove Theorem \ref{SICIAK}. First we prove a pointwise local Weyl law in
the complex domain.

\subsection{\label{ND} Proof of Theorem \ref{SICIAK}(2) }

This follows from Theorem \ref{PTAULWL}
   together with the  following 

  \begin{lem}
\label{COMPARISON} \cite{Z4}  For  any $\tau = \sqrt{\rho}(\zeta) >
0$, and for any $\delta > 0$, 
$$ 2 \sqrt{\rho}(\zeta)  - \frac{\log |\delta|}{\lambda}+ O(\frac{\log \lambda}{\lambda})
 \leq \frac{1}{\lambda} \log \Pi_{[0,   \lambda]}(\zeta,
\bar{\zeta}) \leq 2 \sqrt{\rho}(\zeta)  +  O(\frac{\log \lambda}{\lambda}) $$ hence
$$ \lim_{\lambda \to \infty} \frac{1}{\lambda} \log  \Pi_{[0,  \lambda]}(\zeta, \bar{\zeta}) = 2 \sqrt{\rho}(\zeta). $$
\end{lem}

\begin{proof} 

For the upper bound, we use that

\begin{equation}\label{UB}  \begin{array}{lll}
 \Pi_{[0, \lambda]}(\zeta, \bar{\zeta})  & \leq &  e^{ 2 \lambda
\sqrt{\rho(\zeta)} } \sum_{j: \lambda_j \in [0, \lambda]} e^{- 2
\sqrt{\rho(\zeta)} \lambda_j} |\phi_{\lambda_j}^{\C}(\zeta)|^2 \\ && \\
& = &   e^{ 2 \lambda
\sqrt{\rho(\zeta)} } \;  P_{]0, \lambda]} (\zeta, \bar{\zeta}). \end{array}.
\end{equation} We then take $\frac{1}{\lambda} \log$ of both sides and apply Theorem  \ref{PTAULWL}
to conclude the proof.

The lower bound is subtler for reasons having to do with the distribution of eigenvalues (see the Remark below).
It is most natural to prove two-term Weyl asymptotics for $P_{[0, \lambda]}(\zeta, \bar{\zeta})$ and
to deduce Weyl asymptotics for short spectral intervals $[\lambda, \lambda + 1]$.
 But that requires an analysis of the singularity of the trace of the complexified wave gropup for longer
times than a short interval around $t = 0$ and we postpone the more refined analysis until \cite{Z1}.

Instead we use the longer intervals $[(1 - \delta) \lambda, \lambda]$ for some $\delta > 0$.   We clearly have
\begin{equation}\label{LB}   e^{ 2 (1 - \delta) \lambda
\sqrt{\rho(\zeta)} } \sum_{j: \lambda_j \in [(1 - \delta)\lambda, \lambda]}  e^{- 2
\sqrt{\rho(\zeta)} \lambda_j} |\phi_{\lambda_j}^{\C}(\zeta)|^2
\leq \Pi_{[0, \lambda]}(\zeta, \bar{\zeta}) 
\end{equation}
By Theorem \ref{PTAULWL},
$$ \begin{array}{lll} \sum_{j: \lambda_j \in [ (1 - \delta) \lambda,\lambda]}  e^{- 2
\sqrt{\rho(\zeta)} \lambda_j} |\phi_{\lambda_j}^{\C}(\zeta)|^2 &= & 
 P_{[0, \lambda]}(\zeta, \bar{\zeta})  - P_{[0, (1 - \delta) \lambda]}(\zeta, \bar{\zeta}) \\ && \\
& = & C_n(\tau)  [ 1 -(1 - \delta)^n ] \lambda^{\frac{n+1}{2}} + O(\lambda^{\frac{n-1}{2}} )  \end{array}$$
Taking $\frac{1}{\lambda} \log$ then gives
$$\begin{array}{lll} \frac{1}{\lambda} \log \Pi_{[0, \lambda]}(\zeta, \bar{\zeta} )
\geq 2 (1 - \delta) \sqrt{\rho}(\zeta) -
\frac{|\log \delta|}{\lambda}   + O( \frac{\log \lambda}{\lambda}).  \end{array}$$
It follows that for all $\delta > 0$,
$$\liminf_{\lambda \to \infty}  \frac{1}{\lambda} \log \Pi_{[0, \lambda]}(\zeta, \bar{\zeta} ) 
\geq 2 (1 - \delta) \sqrt{\rho}(\zeta). $$
The conclusion of the Lemma follows from the fact that the left side is independent of $\delta. $

\end{proof}

\begin{rem} The problematic issue in the lower bound is the width of $I_{\lambda}$. If $(M, g)$ is a Zoll manifold, the
eigenvalues of $\sqrt{\Delta}$ form clusters of width $O(\lambda^{-1})$ around an arithmetic progression
$\{k + \frac{\beta}{4}\}$ for a certain Morse index $\beta$. Unless the intervals $I_{\lambda}$ are carefully
centered around this progression, $P_{I_{\lambda}}$ could  be zero. Hence we must use long spectral intervals
if we do not analyze the long time behavior of the geodesic flow; for short ones no general lower bound exists.

\end{rem}

\subsection{Proof of Theorem \ref{SICIAK} (1)}
\begin{proof} We need to show that
$$\Pi^{\C}_{I_{\lambda}}(\zeta, \bar{\zeta}) = \sup \{|\phi(\zeta)|^2: \phi = \sum_{j: \lambda_j \in I}  a_j \phi_{\lambda_j}^{\C}, \;\; ||a|| =  1\}. $$
We define the `coherent state',
$$\Phi_{\lambda}^z(w) = \frac{\Pi^{\C}_{I_{\lambda}}(w, \bar{z})}{\sqrt{\Pi^{\C}_{I_{\lambda}}(z, \bar{z})}}, $$
satisfying,
$$\Phi_{\lambda}^z(w) = \sum_{j: I_{\lambda}} a_j \phi^{\C}_j(w) , \;\;\; a_j = \frac{\overline{\phi^{\C}_j(\zeta)} }{\sqrt{\Pi^{\C}_{I_{\lambda}}(z,
\bar{z})}},\;\;\; \sum_j |a_j|^2 = 1. $$ Hence,
$\Phi_{I_{\lambda}}^{\zeta} $ is a competitor for the sup and
since $|\Phi_{I_{\lambda}}^{\zeta}(\zeta)|^2
 = \Pi_{I_{\lambda}}(\zeta, \bar{\zeta}) $
one has
$$\Pi^{\C}_{I_{\lambda}}(\zeta, \bar{\zeta}) \leq \sup \{|\psi(\zeta)|^2: \psi = \sum_{j: \lambda_j \in I}  a_j \phi_j^{\C}, \;\; ||a|| = 1\}. $$

On the other hand, by the Schwartz inequality for $\ell^2$, for
any $\psi = \sum_{j: \lambda_j \in I}  a_j \phi_j^{\C}$ one has
$$| \sum_{j: \lambda_j \in I}  a_j \phi_j^{\C}|^2 = |\langle a, \psi \rangle|^2 \leq ||a||^2 \sum |\phi_j^{\C}|^2 = \Pi_{I_{\lambda}}(\zeta, \bar{\zeta})$$
and one has
$$\Pi^{\C}_{I}(\zeta, \bar{\zeta}) \geq \sup \{|\psi(\zeta)|^2: \psi= \sum_{j: \lambda_j \in I}  a_j \phi_j^{\C}, \;\; ||a|| = 1\}. $$

\end{proof}

\begin{rem}
Since $N(I_{\lambda}) \sim \lambda^{m-1}$,
$$\frac{1}{\lambda} \log \Pi_{I_{\lambda}}(\zeta,
\bar{\zeta}) =  \frac{1}{\lambda } \log \left( \sum_{j: \lambda_j
\in I_{\lambda}}  |\phi_{\lambda_j}^{\C}(\zeta)|^2 \right) =
\max_{j: \lambda_j \in I_{\lambda}}\{ \frac{1}{\lambda } \log
|\phi_{\lambda_j}^{\C}(\zeta)|^2 \} + O(\frac{\log
\lambda}{\lambda}).
$$

We recall (see \cite{Z3}) that a sequence of eigenfunctions is
called ergodic if  $\langle A \phi_j, \phi_j \rangle \to
\frac{1}{\mu(S^*_g M)} \int_{S^*_g M} \sigma_A d\mu$. The
complexified eigenfunctions then  satisfy $\frac{1}{\lambda_j}
\log |\phi_j(\zeta)| \to \sqrt{\rho}(\zeta)$. It follows that
ergodic eigenfunctions are asymptotically maximal, i.e. have the
same logarithmic asymptotics as $\Phi_M^{\lambda}$.
\end{rem}

\subsection{\label{GENSICIAK} Remarks on more general extremal PSH functions}

We can define a more general Siciak extremal function of a subset
$E \subset M_{\tau}$ by,
$$ \Phi_E^{\lambda} (z) = \sup \{|\psi(z)|^{1/\lambda} \colon
\psi \in \hcal_{\lambda};   \|\psi \|_E \le 1 \}, $$ and
$$\Phi_E(z) = \sup_{\lambda} \Phi_E^{\lambda}(z). $$
It would be interesting to determine this function and the
associated  equilibrium measure of $E$, i.e.  Monge-Amp\`ere mass
of $V_E^*$.

This is of interest even when $E \subset M$ (i.e. is totally
real).  Suppose that instead of orthonormalizing the
eigenfunctions $\phi_j$ on $M$, we  orthonormalize them on a ball
$B \subset M$. Let $\{ \phi_{\lambda_j}^B(x)\}$ be the resulting
orthonormal basis.  We have simply changed the inner product to
$\int_B f_1 f_2 dV_g$. We then obtain a spectral projections
kernel
\begin{equation} \Pi^B_{[0, \lambda}(x, y) : = \sum_{j: \lambda_j \leq \lambda} \phi_{\lambda_j}^B(x) \phi_{\lambda_j}^B(y). \end{equation}
The growth of $\Pi^B_{[0, \lambda]}(\zeta, \bar{\zeta})$
determines doubling estimates for eigenfunctions.  Its exponential
growth rate should be that of the associated pluri-complex Green's
function
 $\log \Phi_B (z) = \lim_{\lambda \to \infty} \frac{1}{\lambda} \log  \Pi^B_{[0,
\lambda}(\zeta, \bar{\zeta})$. It would be interesting to
determine this  analogue of $\sqrt{\rho}$. Its Monge-Amp\`ere mass
should concentrate on $B$, so should be the metric delta-function
on $B$.

\section{Analytic continuation of eigenfunctions}

In this section, we briefly review some results about analytic continuations of eigenfunctions
to Grauert tubes and then prove Proposition  \ref{TRIPLEPROD}.  A more detailed analysis will appear in \cite{Z1,Z5}.

A  function $f$ on a real analytic manifold $M$ is
real analytic, $f \in C^{\omega}(M)$,  if and only if it satisfies the Cauchy estimates
\begin{equation} \label{CEST} |D^{\alpha} f(x) | \leq K \; L^{|\alpha|}
\alpha! \end{equation} for some $K, L > 0$. In place of all
derivatives it is sufficient to use powers of $\Delta$. In the
language of Baouendi-Goulaouic \cite{BG,BG2,BG3}, the Laplacian of
a compact real analytic Riemannian  manifold  has the property of
iterates, i.e. the real analytic functions are precisely the
functions satisfying Cauchy estimates relative to $\Delta$,
\begin{equation}\label{BAOU} C^{\omega}(M)  = \{u \in
C^{\infty}(M):  \exists L > 0, \; \forall k \in {\bf N}, \;
||\Delta^k u||_{L^2(M)} \leq L^{k + 1} (2 k)! \}.
\end{equation}

It is classical that  all of the eigenfunctions extend holomorphic to a fixed Grauert tube. 
\begin{theo} (Morrey-Nirenberg Theorem) Let $P(x, D)$ be an elliptic differential operator in $\Omega$ with coefficients
which are analytic in $\Omega$. If $u \in \dcal'(\Omega)$ and $P(x, D) u = f$ with $f \in C^{\omega}(\Omega)$,
then $u \in C^{\omega}(\Omega). $ \end{theo}
The proof shows that the radius of convergence of the solution is determined by the radius of
convergence of the coefficients.

 In Theorem 2 of \cite{BG2} and
Theorem 1.2 of \cite{BGH} it is proved  that the operator $\Delta$
has the iterate property if and only if, for all $b > 1$, each
eigenfunction extends holomorphically to some Grauert tube
$M_{\tau}$ and satisfies
\begin{equation} \label{BGEST} \sup_{z \in M_{\tau}} |\phi_{\lambda_j}^{\C}(z)|
\leq b^{\lambda_j} \sup_{x \in M} |\phi_{\lambda_j}(x)|.
\end{equation}
The concept of Grauert was not actually used in these articles, so
the relation between the growth rate and the Grauert tube function
was not stated.  But it again shows that all eigenfunctions extend to some fixed Grauert tube.

\subsection{\label{MAX} Maximal holomorphic extension}

The question then arises if all eigenfunctions  extend to the maximal Grauert tube allowed by the geometry
as in Definition \ref{MAXGRAU}. We conjectured in the introduction that this does hold, and
now explain how it should follow from known theorems on extensions of holomorphic solutions
of holomorphic PDE across non-characteristic hypersurfaces. 

\begin{theo} \cite{Zer,Ho3,BSh}  Let $f$ be analytic in the open set $Z \subset \C^n$ and suppose
that $P(x, D) u = f$ in the open set $Z_0 \subset Z$. If $z_0 \in Z \cap \partial Z_0$ and
if $Z_0$ has a $C^1$ non-characteristic boundary at $z_0$, then $u$ can be analytically continued
as a solution of $P(x, D)  u = f$ in a neighborhood of $z_0$. \end{theo}

The idea of the proof is to rewrite the equation as a Cauchy problem with respect to the non-characteristic
hypersurface  and to apply the Cauchy
Kowaleskaya theorem. To employ the theorem we need to verify that the hypsurfaces $\partial M_{\tau}$
are non-characteristic for the complexified Laplacian $\Delta_{\C}$, i.e.  that
$\sum_{i,j} g^{i j}(\zeta) \frac{\partial \sqrt{\rho}}{\partial \zeta_i}  \frac{\partial \sqrt{\rho}}{\partial \zeta_j}
\not= 0. $ To prove this, we observe that in the real domain $g(\nabla r^2, \nabla r^2) = 4 r^2$, an identity
that was used in \eqref{GAMMA}. In this formula $r^2 = r^2(x, y)$ and we differentiate in $x$. We now analytically
continue the identity in $x \to \zeta, y, \to \bar{\zeta}$ and differentiate only with the holomorphic
derivatives $\frac{\partial}{\partial \zeta_j}$. From \eqref{rhoeq},  we get $$g_{\C}(\partial r_{\C}^2(\zeta, \bar{\zeta}),\partial
r_{\C}^2(\zeta, \bar{\zeta})) = -4 r_{\C}^2(\zeta, \bar{\zeta}) =  \rho(\zeta, \bar{\zeta})  > 0.  $$

Hence the Theorem applies and we can analytically continue eigenfunctions across any point of any $\partial M_{\tau}$
for $\tau < \tau_{g}$, the maximal radius of a Grauert tube  in which  the coefficients of $\Delta_{\C}$ are defined and holomorphic. We can take the union
of the open sets where $\phi_j^{\C}$ has a holmomorphic extension to obtain a maximal
domain of holomorphy. If it fails to be $M_{\tau_g}$ there exists a point $\zeta$ with $\sqrt{\rho}(\zeta) <
\tau_g$ so that $\phi_j^{\C}$ cannot be holomorphically extended across $\partial M_{\tau}$ at
$\zeta$. This contradicts the Theorem above and shows that the maximal domain must be $M_{\tau_g}.$

\subsection{Triple inner products of eigenfunctions: Proof of
Proposition \ref{TRIPLEPROD}}

We start with the identity,
\begin{equation} \int_M \phi_{\lambda_j} \phi_{\lambda_k}^2 dV_g =
e^{-  \tau \lambda_j}  \langle e^{\tau \sqrt{\Delta}}
\phi_{\lambda_j}, \phi_{\lambda_k}^2 \rangle,
\end{equation} and then  choose the largest value of $\tau$ for
which $ e^{\tau \sqrt{\Delta}} \phi_{\lambda_j},  e^{\tau
\sqrt{\Delta}} \phi_{\lambda_k} ^2 \in W^s(M)$ for some $s \in
\R$. Since $$\langle e^{\tau \sqrt{\Delta}} \phi_{\lambda_j},
\phi_{\lambda_k}^2 \rangle = \langle \phi_{\lambda_j}, e^{\tau
\sqrt{\Delta}} \phi_{\lambda_k}^2 \rangle, $$ the assumption that
$e^{\tau \sqrt{\Delta}} \phi_{\lambda_k}^2 \in W^s(M)$ implies
that $$\int_M \phi_{\lambda_j} \phi_{\lambda_k}^2 dV_g \leq C e^{-
\tau \lambda_j}  ||\phi_{\lambda_j}||_{W^{-s}} \leq ||e^{\tau
\sqrt{\Delta}} \phi_{\lambda_k}^2 ||_{W^{s}}\; \lambda_j^{s} e^{-
\tau \lambda_j} .
$$

To complete the proof it suffices to show that $ e^{\tau
\sqrt{\Delta}} \phi_{\lambda_j} \in W^s(M)$ and $ e^{\tau
\sqrt{\Delta}} \phi_{\lambda_k} ^2 \in W^s(M)$ for some $s \in \R$
as long as $\tau < \tau_{an}(g)$. This is obvious for all $\tau$
for $\phi_{\lambda_j}$ since $e^{\tau \sqrt{\Delta}}
\phi_{\lambda_j}  = e^{\tau \lambda_j} \phi_{\lambda_j}.$  To see
that it  also holds  for $\phi_{\lambda_k}^2$, we note that the
analytic continuation operator $\acal (\tau)$ is given by
\begin{equation} e^{\tau \sqrt{\Delta}} f = \left(U_{\C}(i \tau)
\right)^{-1} \acal(\tau) f. \end{equation} Since $U_{\C}(i \tau)$
is an elliptic Fourier integral operator of finite order by
Theorem \ref{BOUFIO2}, its inverse is an elliptic Fourier integral
of the opposite order. In particular, it is clear that $e^{\tau
\sqrt{\Delta}} f  \in W^s(M)$ for some $s$ if and only if
$\acal(\tau) f \in \ocal^{t} (\partial M_{\tau})$ for some $t$. In
fact, $\acal(\tau) \phi_{\lambda_k}^2$ is real analytic on
$M_{\tau}$ for any $\tau < \tau_{an}(g).$

To go beyond this result, one would need to know the structure of
$\partial M_{\tau_{an}(g)}$ and about the restriction of analytic
continuations of eigenfunctions to it.

\section{\label{ZEROS} Complex zeros of eigenfunctions: Proof of Theorem \ref{ZERORAN}}

The real  distribution of zeros is by definition the measure
supported on the real nodal hypersurfaces  $Z_{\phi_j} = \{x \in
M: \phi_j(x) = 0\}$ defined by
\begin{equation} \langle [Z_{\phi_j}], f \rangle =
\int_{Z_{\phi_j}} f(x) d {\mathcal H}^{n-1}, \end{equation} where
$d{\mathcal H}^{n-1}$ is the $(n-1)$-dimensional Haussdorf measure
 induced by the Riemannian metric of $(M,
g)$.  The complex nodal hypersurface of an eigenfunction is
defined by
\begin{equation} Z_{\phi_{\lambda}^{\C}} = \{\zeta \in
 M_{\tau}: \phi_{\lambda}^{\C}(\zeta) = 0 \}.
\end{equation}
There exists  a natural current of integration over the nodal
hypersurface, given by
\begin{equation}\label{ZDEF}  \langle [Z_{\phi_{\lambda}^{\C}}] , \phi \rangle =  \frac{i}{2 \pi} \int_{ M_{\tau}} \ddbar \log
|\phi_{\lambda}^{\C}|^2 \wedge \phi =
\int_{Z_{\phi_{\lambda}^{\C}} } \phi,\;\;\; \phi \in \dcal^{ (m-1,
m-1)} (M_{\tau}). \end{equation} In the second equality we used
the Poincar\'e-Lelong formula. The notation $\dcal^{ (m-1, m-1)}
(M_{\tau})$ stands for smooth test $(m-1, m-1)$-forms with support
in $M_{\tau}.$ The nodal hypersurface $Z_{\phi_{\lambda}^{\C}}$
also carries a natural volume form $|Z_{\phi_{\lambda}^{\C}}|$ as
a complex hypersurface in a \kahler manifold. By Wirtinger's
formula, it equals the restriction of $\frac{\omega_g^{m-1}}{(m -
1)!}$ to $Z_{\phi_{\lambda}^{\C}}$. Hence, one can regard
$Z_{\phi_{\lambda}^{\C}}$ as defining  the measure
\begin{equation} \langle |Z_{\phi_{\lambda}^{\C}}| , \phi \rangle
= \int_{Z_{\phi_{\lambda}^{\C}} } \phi \frac{\omega_g^{m-1}}{(m -
1)!},\;\;\; \phi \in C_0(M_{\tau}).
\end{equation}
For background we refer to \cite{Z3}. In that article, we proved:
\begin{theo}\label{ZERORANa}  Let $(M, g)$ be any real analytic compact
Riemannian manifold with ergodic geodesic flow. Then
$$\frac{1}{\lambda_{j_k}} [Z_{\phi_{j_k}^{\C}}] \to  \frac{i}{ \pi} \ddbar |\xi|_g,\;\;
 \mbox{weakly in}\;\; \dcal^{' (1,1)} (B^*_{\epsilon} M), $$
 for a full density  subsequence $\{\phi_{j_k}\}$.
\end{theo}

In this section, we show that the same limit formula is valid for
the {\it entire} sequence of eigenfunctions on higher rank locally
symmetric manifolds studied in \cite{AS}.

\subsection{Plurisubharmonic functions}

We put
\begin{equation} \label{DEFS}\left\{ \begin{array}{l} \phi_{\lambda}^{\epsilon} =
\phi_{\lambda}^{\C}|_{\partial M_{\epsilon}} \in H^2(\partial M_{\epsilon})\\ \\
u_{\lambda}^{\epsilon} : =
\frac{\phi_{\lambda}^{\epsilon}(z)}{||\phi_{\lambda}^{\epsilon}
||_{L^2(\partial M_{\epsilon})}} \in H^2(\partial M_{\epsilon})\\ \\
U_{\lambda}(z) : =
\frac{\phi_{\lambda}^{\C}(z)}{||\phi_{\lambda}^{\epsilon}
||_{L^2(\partial M_{\epsilon})}},\;\;\; z \in \partial
M_{\epsilon}.
\end{array} \right.
\end{equation}

Of these, $U_{\lambda}$ will play the central role.  We note that
 $U_{\lambda}$ is
CR holomorphic on $\partial M_{\tau}$. However, the normalizing
factor $||\phi_{\lambda}^{\epsilon} ||_{L^2(\partial
M_{\epsilon})}^{-1}$ depends on $\epsilon$, so $U_{\lambda} \notin
\ocal(M_{\epsilon}). $

\begin{lem} Let $\{\phi_j\}$ be an orthonormal basis of eigenfunctions
on any compact analytic Riemannian manifold $(M, g)$.  Then for
$\tau < \tau_{an}$, $\{\frac{1}{\lambda_j} \log |U_j|^2\}$ is
pre-compact in $L^1(M_{\tau})$: every sequence  has a convergent
subsequence in $L^1(M_{\tau})$.
\end{lem}

\begin{proof}

As in \cite{Z3}, we  use the following fact about subharmonic
functions (see \cite[Theorem~4.1.9]{Ho}):
\medskip

\begin{itemize}

\item {\it Let $\{v_j\}$ be a sequence of subharmonic functions in
an open set $X \subset \R^m$ which have a uniform upper bound on
any compact set. Then either $v_j \to -\infty$ uniformly on every
compact set, or else there exists a subsequence $v_{j_k}$ which is
convergent in $L^1_{loc}(X)$.}

\item {\it If $v$ is subharmonic and $v_j \to v$ weakly in
$\dcal'(M_{\C})$ then $v_j \to v$ in $L^1$. }

\end{itemize}

We note that  $\frac{1}{\lambda_j} \log |\phi_j^{\C}|$ is
plurisubharmonic and  uniformly bounded above on the Grauert tube.
Therefore, it either tends  to $-\infty$ uniformly on compact sets
of the Gruaert tube  or is pre-compact in $L^1$. The first
possibility is ruled out by the fact that it has the form $U(i
\tau)^{\C} \phi_j$ on $\partial M_{\tau}$. Hence,
$$||\phi^{\C}_j||_{L^2(\partial M_{\tau})} = e^{2 \tau \lambda_j} \langle
U(i \tau)^* U(i \tau) \phi_j, \phi_j \rangle_{L^2(M)} \geq e^{2
\tau \lambda_j} \lambda_j^{- \frac{m-1}{2}},
$$ by Garding's inequality (\ref{GARDING}). This contradicts
 the hypothesis that
 $\frac{1}{\lambda_j} \log |\phi_j^{\C}|$ tends to zero uniformly
 on all compact sets, i.e. that
 $|\phi_j^{\C}(\zeta)| \leq e^{- \epsilon_{\tau} \lambda_j}$.

\end{proof}

We thus have two different and independent  types of weak limit
problems:
\begin{itemize}

\item Weak limits of the $L^2$-normalized shell functions $U_j$;

\item Weak limits of $\frac{1}{\lambda} \log |u_j|$.

\end{itemize}

\begin{lem} Suppose that $\{\phi_j\}$ is a sequence of
eigenfunctions with a unique limit measure $d\mu$ and suppose that
$d\mu  = \rho d\mu_L + \nu$ where  $\rho \geq C > 0$ and $\nu \bot
\mu_L$. Then $\frac{1}{\lambda_j} Z_{\lambda_j} \to i \ddbar
|\xi|$.
\end{lem}

\begin{proof} We claim that in this case $\frac{1}{\lambda_j} \log |U_j| \to 0.$
Indeed, it is clear that the limsup of the left side is $\leq 0$.
On the other hand, suppose that the limsup is negative on an open
set $U$. Then $\int_U |U_j|^2 \to 0$. This contradicts the
assumption that limit measure has an everywhere positive Liouville
component. The rest of the proof is exactly the same as in
\cite{Z3}.

\end{proof}

To complete the proof of Theorem \ref{ZERORAN}, we recall (from
the introduction) that Theorems 1.8, 1.9, 1.10 of \cite{AS} prove
that for co-compact lattices $\Gamma \subset SL_n(\R)$, any
semi-classical measure has a Haar component of positive weight.
Hence the hypotheses of the Lemmas are satisfied by joint
$\dcal$-eigenfunctions of the locally symmetric spaces.


\begin{thebibliography}{MMM}

\bibitem[AS]{AS} N. Anantharman and L. Silberman,
 A Haar component for quantum limits on locally
symmetric spaces, arXiv:1009.4927.




\bibitem[BG]{BG} M.S. Baouendi and C.  Goulaouic,
 R\'egularit\'e analytique et it\'er\'es d'op\'erateurs elliptiques d\'eg\'en\'er\'es; applications.
  J. Functional Analysis 9 (1972), 208--248.

  \bibitem[BG2]{BG2} M.S. Baouendi and C.  Goulaouic, It\'er\'es d'op\'erateurs elliptiques et prolongement de fonctions propres.
   Hommage au Professeur Miron Nicolescu pour son 70e anniversaire, I. Rev. Roumaine Math. Pures Appl. 18 (1973),
   1495--1501.


\bibitem[BG3]{BG3} M.S. Baouendi and C. Goulaouic, Cauchy problem for analytic pseudo-differential operators.
Comm. Partial Differential Equations 1 (1976), no. 2, 135�189.

\bibitem[BGH]{BGH} M. S. Baouendi, C.  Goulaouic, and B.  Hanouzet,
Caract\'erisation de classes de fonctions $C\sp{\infty }$ et
analytiques sur une vari�t� irr\'eguli�re � l'aide d'un
op\'erateur diff\'erentiel. J. Math. Pures Appl. (9) 52 (1973),
115--144.

\bibitem[BT]{BT} E. Bedford and B.A.  Taylor,
The Dirichlet problem for a complex Monge-Amp\`ere equation.
Invent. Math. 37 (1976), no. 1, 1�44

\bibitem[BT2]{BT2} E. Bedford and B. A. Taylor,
A new capacity for plurisubharmonic functions. Acta Math. 149
(1982), no. 1-2, 1�40.

\bibitem[Be]{Be} P. B\'erard, On the wave equation without
conjugate points, Math.\ Zeit.\ {\bf 155} (1977), 249--276.




\bibitem[BGV]{BGV} N. Berline, E. Getzler and
M.  Vergne, {\it  Heat kernels and Dirac operators}.  Grundlehren
Text Editions. Springer-Verlag, Berlin, 2004.

\bibitem[BR]{BR} J.  Bernstein and A  Reznikov,
 Analytic continuation of representations and estimates of automorphic forms. Ann. of Math. (2) 150 (1999), no. 1, 329--352


\bibitem[BL]{BL} T. Bloom and N. Levenberg,
Weighted pluripotential theory in $\bold C\sp N$.  Amer. J. Math.
125 (2003), no. 1, 57--103.


\bibitem[BS]{BS} T. Bloom and B. Shiffman,
Zeros of random polynomials on $\Bbb C\sp m$. Math. Res. Lett. 14
(2007), no. 3, 469--479.

\bibitem[BSh]{BSh} J. M. Bony and P. Schapira,
Existence et prolongement des solutions holomorphes des
\'equations aux d\'eriv\'ees partielles. Invent. Math. 17 (1972),
95--105.

\bibitem[Bou]{Bou} L. Boutet de Monvel,
Convergence dans le domaine complexe des s\'eries de fonctions
propres.  C. R. Acad.\ Sci.\ Paris S\'er. A-B 287 (1978), no.\ 13,
A855--A856.


\bibitem[BouG]{BouG} L.  Boutet de Monvel and V. Guillemin,
{\it The spectral theory of Toeplitz operators.} Annals of
Mathematics Studies, 99. Princeton University Press, Princeton,
NJ; University of Tokyo Press, Tokyo, 1981.


\bibitem[BoSj]{BoSj} L. Boutet de Monvel and J. Sj\"ostrand, Sur la
singularit\'e des noyaux de Bergman et de Szeg\"o, {\it
Asterisque\/} 34--35 (1976), 123--164.

\bibitem[Br]{Br} H. J. Bremermann,
On a generalized Dirichlet problem for plurisubharmonic functions
and pseudo-convex domains. Characterization of Silov boundaries.
Trans. Amer. Math. Soc. 91 1959 246�276.

\bibitem[BW]{BW} F. Bruhat and H. Whitney,
Quelques propri\'et\'es fondamentales des ensembles
analytiques-r\'eels.
 Comment. Math. Helv. 33 1959 132--160.


\bibitem[DF]{DF} H. Donnelly and C. Fefferman, Nodal sets of eigenfunctions on
Riemannian manifolds, Invent. Math. 93 (1988), 161-183.





\bibitem[DG]{DG} J. J. Duistermaat and V. W.  Guillemin,
The spectrum of positive elliptic operators and periodic
bicharacteristics. Invent.\ Math. 29 (1975), no.\ 1, 39--79.


\bibitem[DH]{DH} J. J. Duistermaat and L. H\"ormander, Fourier integral operators II,  Acta Math. 128 (1972), no. 3-4, 183–269. 



\bibitem[Gar]{Gar} P. R. Garabedian, {\it Partial differential equations}. Reprint of the 1964 original. AMS Chelsea Publishing, Providence, RI, 1998.



\bibitem[GeSh]{GeSh}  I. M. Gelfand and G. E. Shilov, {\it Generalized
Functions, Vol. 1}, Academic Press (1964).

\bibitem[GLS]{GLS} F. Golse, E. Leichtnam, and M. Stenzel,
 Intrinsic microlocal analysis and inversion formulae for
  the heat equation on compact real-analytic Riemannian manifolds.
  Ann. Sci.\ \'Ecole Norm.\ Sup. (4) 29 (1996), no.\ 6, 669--736.


\bibitem[Gr]{Gr} H.  Grauert, \"Uber Modifikationen und exzeptionelle
analytische Mengen, {\it Math.\ Annalen} 146 (1962), 331--368.

\bibitem[GZ]{GZ} V. Guedj and A.  Zeriahi,
 Intrinsic capacities on compact K\"ahler manifolds. J. Geom. Anal. 15 (2005), no. 4, 607--639.

\bibitem[G]{G}  V. Guillemin,
Paley-Wiener on manifolds. In {\it The Legacy of Norbert Wiener: A
Centennial Symposium} (Cambridge, MA, 1994), 85--91, Proc. Sympos.
Pure Math., 60, Amer. Math. Soc., Providence, RI, 1997.




\bibitem[GS1]{GS1} V. Guillemin and M. Stenzel, Grauert tubes and the
homogeneous Monge-Amp\`ere equation. J. Differential Geom. 34
(1991), no.\ 2, 561--570.

\bibitem[GS2]{GS2}  V. Guillemin and M. Stenzel,  Grauert tubes and the
homogeneous Monge-Amp\`ere equation. II. J. Differential Geom. 35
(1992), no.\ 3, 627--641.

\bibitem[H]{H} J. Hadamard, {\it Lectures on Cauchy's problem in linear partial differential equations}. Dover Publications, New York, 1953.

%

%

\bibitem[HL]{HL} Q. Han and F.-H. Lin,  {\it Nodal sets of solutions of elliptic
differential equations}, book in preparation (2008).


\bibitem[HW]{HW} F. R. Harvey and R. O. Wells, Zero sets of non-negative strictly plurisubharmonic functions. Math. Ann. 201 (1973),
165�170.

\bibitem[HS]{HS} B. Helffer and J. Sj\"ostrand, Multiple wells in the semiclassical
limit I, Comm. Partial Differential Equations 9 (1984), 337--408.

\bibitem[Hel]{Hel} S. Helgason, {\it Groups and geometric analysis. Integral geometry, invariant differential operators, and spherical functions.}
 Mathematical Surveys and Monographs, 83. American Mathematical Society, Providence, RI, 2000.



\bibitem[Ho]{Ho}  L. H\"ormander, {\em The Analysis of Linear Partial
Differential Operators\/}, Volumes III-IV, Springer-Verlag Berlin
Heidelberg, 1983.

\bibitem[Ho2]{Ho2} L. H\"ormander, The spectral function of an elliptic operator. Acta Math. 121 (1968),
193--218.

\bibitem[Ho3]{Ho3} L. H\"ormander, {\it Linear Partial Differential Operators}, vol. 116. Springer Grundlehren, Berlin (1976)

\bibitem[IZ]{IZ} C. Itzykson and J.B.  Zuber, Quantum field theory. International Series in Pure and Applied Physics. McGraw-Hill International Book Co., New York, 1980.







rge(1-YALE)
\bibitem[JL]{JL} J. Jorgenson and S. Lang, 
Analytic continuation and identities involving heat, Poisson, wave and Bessel kernels. (English summary) 
Math. Nachr. 258 (2003), 44–70.

\bibitem[J]{J} M. S. Joshi, Complex powers of the wave operator, Portugaliae Math. 54 (1997), 345- 362. 

\bibitem[KM]{KM} S-J. Kan and D.  Ma, On rigidity of Grauert tubes over Riemannian manifolds of constant
curvature. Math. Z. 239 (2002), no. 2, 353--363.


 \bibitem[K]{K} M. Klimek, Pluripotential Theory. London Mathematical Society Monographs. New Series, 6, Oxford University Press, New York, 1991.

\bibitem[KS]{KS}  B. Kr\"otz and  H. Schlichtkrull, 
Holomorphic extension of eigenfunctions.  
Math. Ann. 345 (2009), no. 4, 835–841. 

 \bibitem[Lax]{Lax} P. Lax,   Asymptotic solutions of oscillatory initial value problems. Duke Math. J. 24 1957 627--646.

\bibitem[Leb]{Leb} G. Lebeau,  Fonctions harmoniques et spectre singulier. Ann. Sci. \'Ecole Norm. Sup. (4) 13 (1980), no. 2, 269--291.



\bibitem[LS1]{LS1} L. Lempert and R.  Sz\"oke,
 Global solutions of the homogeneous complex Monge-Amp\`ere equation and
 complex structures on the tangent bundle of Riemannian manifolds.
Math. Ann. 290 (1991), no.\ 4, 689--712.




\bibitem[Lin]{Lin} F.H.  Lin, Nodal sets of solutions of elliptic
and parabolic equations. Comm. Pure Appl. Math. 44 (1991), no. 3,
287--308.


\bibitem[Mar]{Mar} A.
Martinez, Estimates on complex interactions in phase space, Math.
Nachr. 167 (1994), 203--254.


\bibitem[MeSj]{MeSj} A. Melin and J.  Sj\"ostrand,
 Fourier integral operators with complex-valued phase functions.{\it  Fourier integral operators and partial differential equations}
  (Colloq. Internat., Univ. Nice, Nice, 1974), pp. 120�223. Lecture Notes in Math., Vol. 459, Springer, Berlin,
  1975.

\bibitem[Miz]{Miz} S. Mizohata,
Analyticity of the fundamental solutions of hyperbolic systems. J.
Math. Kyoto Univ. 1 1961/1962 327--355.

\bibitem[Miz2]{Miz2} S. Mizohata,
Analyticity of solutions of hyperbolic systems with analytic
coefficients. Comm. Pure Appl. Math. 14 1961 547-559.











\bibitem[R]{R} M. Riesz, L'int\'egrale de Riemann-Liouville et le probl�me de Cauchy.  Acta Math. 81, (1949). 1--223.

\bibitem[RZ]{RZ} Y. A. Rubinstein and S. Zelditch,  Complex zeros of integrable
 eigenfunctions, (in preparation).

\bibitem[Saf]{Saf} Y. Safarov, Asymptotics of a spectral function of a positive
elliptic operator without a nontrapping condition. (Russian)
Funktsional.\ Anal.\ i Prilozhen. 22 (1988), no.\ 3,53--65, 96
translation in Funct.\ Anal. Appl.\ 22 (1988), no.\ 3, 213--223
(1989).

\bibitem[SV]{SV} Y. Safarov and D. Vassiliev, {\it The asymptotic distribution
of eigenvalues of partial differential operators. Translated from
the Russian manuscript by the authors}. Translations of
Mathematical Monographs, 155. American Mathematical Society,
Providence, RI, 1997.

\bibitem[Sar]{Sar} P. Sarnak,  Integrals of products of eigenfunctions. Internat. Math. Res. Notices 1994, no. 6, 251 ff.,



\bibitem[Sic]{Sic} J. Siciak,
Extremal plurisubharmonic functions in $C\sp{n}$. Ann. Polon.
Math. 39 (1981), 175--211.


\bibitem[Sic2]{Sic2} J. Siciak,
On some extremal functions and their applications in the theory of
analytic functions of several complex variables. Trans. Amer.
Math. Soc. 105 1962 322--357.

\bibitem[Sj]{Sj} J. Sj\"ostrand,  Singularit\'es analytiques microlocales.  Ast\'erisque, 95, 1--166,  Soc. Math. France, Paris,
1982.






%

\bibitem[St]{St} E. M. Stein, {\it  Singular integrals and differentiability properties of functions}. Princeton Mathematical Series, No. 30 Princeton University Press, Princeton, N.J. 1970
\bibitem[Sz]{Sz} R.  Szoke, Complex structures on tangent bundles of Riemannian manifolds. Math. Ann. 291 (1991), no. 3, 409--428
%

\bibitem[T]{T} M.E. Taylor, {\it Noncommutative harmonic analysis.} Mathematical Surveys and Monographs, 22. American Mathematical Society, Providence, RI, 1986.

\bibitem[To]{To} J. A. Toth,  Eigenfunction decay estimates in the quantum integrable case. Duke Math. J. 93 (1998), no. 2, 231--255.

\bibitem[TZ]{TZ} J. A. Toth and S. Zelditch, Counting nodal lines which touch the boundary of an analytic domain,  Jour.
Diff. Geom. 81 (2009), 649- 686 (arXiv:0710.0101).

%




\bibitem[Za]{Za} V. P. Zaharjuta, Extremal plurisubharmonic functions, orthogonal polynomials, and the Bernstein-Walsh theorem for
functions of several complex variables (Russian), Ann. Polon.
Math. 33 (1976/77), 137--148.


\bibitem[Z1]{Z1} S. Zelditch,  Pluri-potential  theory on Grauert tubes of  real analytic Riemannian
manifolds, II (in preparation).


\bibitem[Z2]{Z2} S. Zelditch,
 Real and complex zeros of Riemannian random waves, {\it Spectral analysis in geometry and number theory},
  321--342, Contemp. Math., 484, Amer. Math. Soc., Providence, RI, 2009.







\bibitem[Z3]{Z3} S. Zelditch, Complex zeros of real ergodic
eigenfunctions, Invent. Math. 167 (2007), no. 2, 419--443.


  \bibitem[Z4]{Z4}  S. Zelditch, {\it \szego kernels and a theorem of Tian},
IMRN  6 (1998), 317--331.

\bibitem[Z5]{Z5} S. Zelditch, CBMS Lectures (in preparation).

%


\bibitem[Zer]{Zer} M. Zerner, Domain d'holomorphie  des fonctions v\'erifiant une \'equation aux d\'eriv\'ees partielles,  C. R. Acad. Sci. Paris Ser. I Math. 272 (1971) 1646-1648.


\end{thebibliography}
\end{document}